\documentclass[review,onefignum,onetabnum]{siamonline220329}
\usepackage[normalem]{ulem}
\usepackage{lettrine}
\usepackage[utf8]{inputenc}
\usepackage{amsmath}
\usepackage{amsfonts}
\usepackage[english]{babel}
\usepackage{mathrsfs}
\usepackage{graphicx}
\graphicspath{{./Images/}}
\usepackage{setspace}
\usepackage{bigints}
\usepackage[left=2cm,right=2cm,top=3cm,bottom=3cm]{geometry}
\usepackage[normalem]{ulem}
\usepackage{lettrine}
\usepackage{amssymb}
\usepackage[utf8]{inputenc}
\usepackage{amsmath}
\usepackage{calc}

\usepackage{accents}
\usepackage{bm}
\usepackage{xcolor}
\graphicspath{{./Images/}}
\usepackage{setspace}
\usepackage{bigints}
\usepackage{setspace}
\newsiamthm{thm}{Theorem}
\newsiamthm{prop}{Proposition}
\newsiamthm{lem}{Lemma}
\newsiamthm{cor}{Corollary}
\newsiamremark{convention}{Convention}
\newsiamremark{remark}{Remark}
\newsiamremark{example}{Example}

\makeatletter
\newcommand{\bigperp}{%
	\mathop{\mathpalette\bigp@rp\relax}%
	\displaylimits
}
\usepackage{esint}
\newcommand{\bigp@rp}[2]{%
	\vcenter{
		\m@th\hbox{\scalebox{\ifx#1\displaystyle2.1\else1.5\fi}{$#1\perp$}}
	}%
}
\makeatother

\newcommand{\ol}{\overline}

\newcommand{\ad}{\text{ad}}

\def\rr{\mathbb{R}}
\definecolor{mydarkred}{RGB}{180, 10, 10}

\usepackage{cancel}
\def\bel{\begin{equation}\label}
	\def\eeq{\end{equation}}
\def\ds{\displaystyle}

\def\U{\Gamma}

\numberwithin{equation}{section}

\title{Higher order necessary conditions for optimal controls not ranging in the interior}

\usepackage{fancyhdr}
\author{Francesca Angrisani\footnote{Sorbonne, Jacques Louis Lions Laboratory (LJLL), Paris, France. (\email{francesca.angrisani@sorbonne-universite.fr}).} \,\,and\,\, Franco Rampazzo\footnote{Department of Mathematics T. Levi-Civita, University of Padova, Italy. (\email{rampazzo@math.unipd.it}).}}

\begin{document} 
	\nolinenumbers
	\maketitle

	\normalsize

	\noindent
	
	\nolinenumbers
	
\begin{abstract}  Goh's and Legendre-Clebsch necessary conditions for optimal controls of affine-control systems are usually established under the hypothesis that the minimizing control lies in the interior of the control set $U$. In this paper we investigate the possibility of establishing  Goh's and Legendre-Clebsch necessary conditions without  this assumption, so that even control sets with empty interiors or optimal controls touching  the boundary of $U$  can be taken into consideration.
\end{abstract}

	\section{Introduction}
	
Besides the classical, first order,  Maximum Principle, higher order necessary conditions  are classically considered for optimal controls of ordinary differential equations. In particular, when the differential system is control-affine and the minimum problem is of the form 
		$$\qquad\left.
	\left.
	\begin{array}{l}
	 \ds	
		\min\Psi(x(T))\\ 

			\displaystyle \frac{dx}{dt}=f(x(t))+\sum\limits_{i=1}^m g_i(x(t))u^i(t), \quad  a.e. \,\,\, t \in [0,T],\\
			\displaystyle x(0) =\hat{x},\qquad\qquad x(T) \in \mathfrak{T}.
 \end{array}\right.\right.$$ 
---where the vector fields $f, g_i$ and the cost $\Psi$ are sufficiently regular, the target $\mathfrak{T}$ is any subset of the state space, and the controls take values in a subset $U\subset\rr^m$, $m\geq 1$---	Goh conditions and  Legendre-Clebsch conditions for an optimal control-trajectory pair $(\ol{u},\ol{x})$ read
\bel{introGLC}
\begin{array}{l}
\ds	{ (a)} \quad p(t)\cdot [g_i,g_j](\ol{x}(t)) =0  \quad\forall 1\leq i<j\leq m\qquad\text{and}\\\\
\ds	{ (b)} \quad p(t)\cdot [f,g_i](\ol{x}(t)) =0  \quad\forall 1\leq i\leq m\qquad \forall t\in[0,T],
	\end{array}
\eeq
	respectively, where  $p(\cdot)$ is the same  adjoint map  whose existence is postulated in  the first order Maximum Principle.  Furthermore, in the special case when $m=1$ (and $g:=g_1$),  they are complemented by the third order
	Legendre-Clebsch condition
	\bel{introLC3}
	0 \geq p(t) \cdot [g,[f,g]](\ol{x}(t))  \qquad \forall t\in[0,T]
	\eeq
	 Let us remark that  the  crucial hypothesis (S) below is made on the optimal control $\ol{u}$ in order to obtain the higher order conditions \eqref{introGLC} and \eqref{introLC3}:\\
	 
	  \noindent {(S)}\hspace{2truecm}\hspace{1truecm} {\it   $\ol{u}(t)\in int(U)$ for almost every $t\in [0,T]$}, \\
	  
	  \noindent where  $int(U)$ denotes the  interior of $U$. An optimal control verifying  hypothesis (S) is often called {\it singular}, and the reason why one assumes  the singularity hypothesis (S) consists in the following fact: on a left neighbourhood of almost any time $t$, the primitive  of a singular  control can be perturbed by means of suitable infinitesimal continuous paths having the effect of producing infinitesimal variations of the corresponding trajectory $\ol{x}$ in the directions of the Lie brackets involved in \eqref{introGLC} and \eqref{introLC3}.  
	  
	  The present paper is an attempt to obtain  higher order optimality conditions of type \eqref{introGLC} and \eqref{introLC3} in situations where the singularity assumption (S) is {\it not} verified. In order to replace it, for a given control $u:[0,T]\to U$, a time $t\in ]0,T[$, and $i\in \{1,\dots,m\}$, we introduce the notion stating that {\it $u$ is $i$-{balanced}  at $t$}:  this means that there must exist two positive numbers $\alpha^i, \beta^i$ such that $\{u(s)+ \alpha^i\mathbf{e}_i, u(s)- \beta^i\mathbf{e}_i, \}\subset U$, for any $s$ in a left  neighbourhood of $t$. In the case of Goh conditions \eqref{introGLC}-{(a)},  for  a given pair   $(i,j)$, $0\leq i<j\leq m$,  we replace (S) with  weaker hypothesis stating that the control $\ol{u}$ has to be both $i$-balanced and $j$-balanced at almost every $t\in [0,1]$: namely, there must exist four positive numbers  $\alpha^i, \alpha^j, \beta^i, \beta^j$ such that  $\{u(s)+ \alpha^i\mathbf{e}_i, u(s)+ \alpha^j\mathbf{e}_j, u(s)- \beta^i\mathbf{e}_i,  u(s)-\beta^j\mathbf{e}_j \}\subset U$ for any $s$ in a left neighbourhood of $t$. Similarly, in the case of  Legendre-Clebsch  conditions \eqref{introGLC}-{(b)} and \eqref{introLC3},  for a given $i$ we assume that $\ol{u}$ has to be {\it $i$-balanced}. Notice that under our weakened hypotheses  Goh conditions \eqref{introGLC} [resp.  Legendre-Clebsch conditions \eqref{introLC3}] might be valid only for certain pairs $(i,j)$ [resp. for certain $i$], while failing for the other pair of indexes [resp. the other indexes].

The paper is organized as follows. In Section \ref{main th sect} we state the problem and present the main result (Theorem \ref{TeoremaPrincipale}). Moreover we give some instances of controls that are  {$(i,j)$-fit} or { $(i)$-fit} but are not in the interior of $U$. Additionally, 
in a toy example (see Example \eqref{ESEMPIO})  our  higher order conditions allow to establish, in a  case where the  control set has  empty interior, that a certain control map  verifying the first order, Pontryagin Maximum Principle is in fact not optimal.  In Section \ref{Lie gen} we construct variations of a process  $(\ol{u},\ol{x})$ in the direction of specific Lie brackets under   {$(i,j)$-fit} or { $(i)$-fit} hypotheses. Section \ref{approx.sect}, which deals with approximation of trajectories through products of  exponential maps corresponding to Lie brackets,    is not at all original and is put there  with  the only purpose of adding self-contained character to the paper. Section \ref{proof3.6} is devoted to the proof of Theorem \ref{eccolestime}, which concerns  infinitesimal variations.  In Section \ref{setsepsec} we exploit set-separation issues to deduce the proof of Theorem \ref{TeoremaPrincipale}. A remark    concerning a possible generalization of the presented result to the case of less smooth control systems concludes the paper.
	  
\subsection{Remarks on the notation}	 
 \,\,
 
{\bf Lie brackets}	The Lie bracket $[X,Y]$ (on a differential manifold) of two vector fields $X,Y$ at a point $x$,  is defined as $x\mapsto[X,Y](x):=DY(x)\cdot X(x)-DX(x)\cdot Y(x)$. This is a definition given in coordinates, but $[X,Y]$ is in fact a vector field, so in particular it is an intrinsic object.

{	{\bf Differential equations.} If we have uniqueness of the solution  solution $x(\cdot)$ on some interval $[0,T]$  to a Cauchy problem  $\dot x(t) = F(t,x(t)),\,x(0)=\tilde x$,}
  we will use the exponential notation $ x(t) =e^{\int_0^t F}(\tilde x)$.
	In the particular case of an autonomous vector field $F = F(x)$,  we will use the notation $ x(t) = e^{tF}(\tilde x)$. 

	{\bf Areas.} In order to recognize a  certain geometrical meaning in some coefficients appearing in exponential maps of Lie brackets (see Sect. \ref{Lie gen}),
		for any   plane,  closed,  curve $(C^1,C^2):[a,b] \to \rr^2 $ of class $W^{1,2}$, we  define  ---according to  Green's Theorem--- the {\rm signed area  ${Area}(C^1,C^2)$ of the region {\it encircled} by $(C^1,C^2)$}  as the quantity
		\bel{area}
		{Area}(C^1,C^2) : = \left\langle C^1, \frac{dC_2}{ds}\right\rangle_{L^1}=\bigintss_{a}^{b} C^1(s)\frac{dC_2}{ds}(s) ds.
		\eeq
		 
		The $Area$ operator is invariant for positive reparameterization of the curve. Namely,
		if $\phi:[\alpha,\beta]\to[a,b]$ is a $W^{1,1}$ increasing  map such that $\phi(\alpha)=a, \phi(\beta) = b$, and we set $(\hat C_1, \hat C_2):=( C_1,C_2)\circ\phi$, then ($(\hat C_1, \hat C_2)$ is a closed curve,with $(\hat C_1, \hat C_2)([a,b])= (C_1, C_2)([\alpha,\beta])$, and) one has
		$
		{Area}(\hat C_1, \hat C_2) = {Area}(C_1, C_2).
		$

		\section{The main result}\label{main th sect} We  will be concerned with the optimal control problem
					$$(P)\qquad\left\{
		\left.
		\begin{array}{l}
			\quad \ds	
			\min\Psi(x(T)),\\ \\
			
			\begin{cases}
				
				\displaystyle \frac{dx}{dt}=f(x(t))+\sum\limits_{i=1}^m g_i(x(t))u^i(t), \quad  a.e. \,\,\, t \in [0,T],\\
				\displaystyle x(0) =\hat{x},\qquad x(T) \in \mathfrak{T}.
		\end{cases} \end{array}\right.\right.$$ 
		{ 		where, for some integers $n,m$, \begin{itemize}
				\item  the state $x$ takes values in $\mathbb R^n$;
			\item 	the vector fields $f,g_1,\ldots,g_m$ as well as the {\it cost} $\Psi$ are assumed to be of class $C^{1}$;
				\item the control  $u$  takes values in  a  subset  $U\subseteq\mathbb{R}^m$ (possibly with empty interior);
				\item $\mathfrak{T}$ is a subset of $\mathbb{R}^{n}$, called the  {\it target};
				\item  by {\it  process}  we mean  a  pair  $(u,x)$
				such that $\displaystyle u\in L^1([0,T],U)$
				and 
				$x\in W^{1,1}([0,T],\mathbb{R}^n)$ is the corresponding  (Carath\'eodory)
				 solution of the above Cauchy problem; 
				 \item a  process  $(u,x)$ is  called   {\it feasible}  as soon as $x(T)\in \mathfrak{T}$;
				\item 	the  minimization is performed over the set of the {\it feasible processes} $(u,x)$.
			\end{itemize}
		
		\begin{definition} A { feasible process} $(\hat u,\hat x)$ is called a local  weak minimizer of problem {\rm (P)} if 
		there exists an $L^1$ neighbourhood $\mathcal U$  of $\hat u$ such that 
	$$
\Psi(\hat x(T))\leq \Psi( x(T))
$$ 
for all $x(\cdot) \in  W^{1,1}([0,T],\mathbb{R}^n)$ such that $(x,u)$ is a feasible  process with $u \in \mathcal U$. 

 A { feasible process} $(\hat u,\hat x)$ is called a local  strong  minimizer of problem (P) if 
there exists an $C^0$ neighbourhood $\mathcal V_{\mathcal C}$  of $\hat x$ such that 
$$
\Psi(\hat x(T))\leq \Psi( x(T))
$$ 
for all  feasible  processes $(x,u)$ such that  $x \in \mathcal V_{\mathcal C}$.   

 \end{definition}
			Clearly any local  strong  minimizer  is also a local  weak  minimizer. 
	 
%

\vskip1truecm

In order to state Lie brackets-including higher order conditions  for minima,    we need the notion of $i$-balanced control $u$ at a time $t$.
	Let $\mathbf{e}_0,\ldots,\mathbf{e}_m$  be the canonical basis of $\mathbb{R}\times \mathbb{R}^m$.

%
%

 \begin{definition}\label{fit2}   If  $r\in\{1,\dots m\}$, $ t\in ]0,T[$, a  control  $u:[0,T]\to U$ is called { \bf $r$-balanced at $t$} if  there exists  $\delta\in]0,\min\{t,T-t\}[$ and  $\alpha^r, \beta^r>0$ such that  $$\{u(s)+\alpha^r\mathbf{e}_r, u(s)-\beta^r\mathbf{e}_r   \}\subset U\qquad \text{for a.e.} \,\,s\in ]t-\delta,t[.$$ Furthermore, a  control  $u:[0,T]\to U$ is called { \bf $r$-balanced  a.e.} if there exists a full-measure subset $\Lambda^r_{u}\subseteq ]0,T[$ such  $u$  is $r$-balanced  at all {$t \in \Lambda^r_{u}$}. \end{definition}

	\begin{theorem}[{A higher order  maximum principle}] \label{TeoremaPrincipale}
		Let $(\ol{u},\ol{x})$ be a  local  weak minimizer  of problem {\rm (P)}, and let us set 
 $$H(x,p,u):=p\cdot\Big(f(x)+\sum\limits_{i=1}^m g_i(x)u^i\Big).$$ Furthermore, let $C$ be the Boltyanski approximating cone to the target ${\mathfrak{T}}$ at $\ol{x}(T)$.
 
		Then there exist multipliers $(p,\lambda) \in AC\Big([0,T],(\mathbb{R}^n)^*\Big) \times \mathbb{R^*} $, with $\lambda\geq 0$
		  such that the following properties are satisfied:
		\begin{itemize}	
		\item[\rm\bf i)]{\sc(non-triviality)}
		$$(p,\lambda)\neq (0,0)	$$
			
		\item[\rm\bf ii)]{\sc(Adjoint equation)} $$\frac{dp}{dt} = -  \frac{\partial H}{\partial x} (\ol{x}(t),p(t),\ol{u}(t))\qquad for \,\,a.e.\quad t\in [0,T] \,\,\,  $$ 
		\item[\rm\bf iii)]{\sc(non-transversality)} $$p({T})\in -\lambda   \frac{\partial\Psi}{\partial x}(\ol{x}({T}))-C^\perp  \,\,$$ 
			\item[\rm\bf iv)]{\sc(Maximum Principle)}\\  $$\max_{u\in U} H(x(t),p(t),u) =  H(x(t),p(t),\ol{u}(t))\qquad  for\,\, a.e.\quad t\in [0,T] , \,\,\,\forall u\in U $$
		\end{itemize}
Moreover 	$(p,\lambda)$ can be chosen so that the following further three properties are verified: 	\begin{itemize}
			\item[\rm\bf v)]{$(i,j)$-\sc(Goh condition)} { If  $i,j\in\{1,\dots m\}, i < j$, and the control  $\ol{u}$ is both $i$-balanced and $j$-balanced a.e., then }\\ 
			\bel{gohij}{0 = p(t)\,\cdot [g_i,g_j](\ol{x}(t))  \qquad  \forall t\in[0,T] }
			\eeq 
			\item[\rm\bf vi)]{$i$-\sc(Legendre\textendash{}Clebsch  condition of step $2$)}  { If  $i\in\{1,\dots m\}$ and the control  $\ol{u}$ is  $i$-balanced a.e., then }\\ 	\bel{goh12}
			{0 =  p(t)\,\cdot [f,g_i](\ol{x}(t))  \qquad \forall t\in[0,T]} 
			\qquad \footnote{Because of the antisimmetry of the Lie bracket, \eqref{gohij} and \eqref{goh12} are equivalent  to $0 = p(t)\,\cdot [g_j,g_i](\ol{x}(t))$ and  $0=p(t)\,\cdot [g_i,f](\ol{x}(t)) $, respectively.}
			\eeq
			\item[\rm\bf vii)]{\sc(Legendre\textendash{}Clebsch condition of step $3$ with $m=1$)}  {If $m=1$, $f,g:=g_1$ are of class $C^{2}$ around $\ol{x}([0,T])$, and the control 	$\ol{u}$  is  $1$-balanced, then }\\ 
		\bel{LC32}
		{0\ge  p(t)\,\cdot [g,[f,g]](\ol{x}(t))  \qquad \forall t\in[0,T] }
		\eeq	
			\end{itemize}
		\end{theorem}
	
	\begin{remark}
		If a control $u$ takes values in the interior of $U$, then it is   $i$- balanced  a.e. for all  $i\in\{1,\dots m\}$. Because of this, Goh and Legendre-Clebsch conditions in their classical form are a particular case of Theorem \ref{TeoremaPrincipale}.
	\end{remark}
	\begin{example} If 
		$U:= \mathbb{N}^3$, then any control map $u=(u^1,u^2,u^3):[0,T]\to U$ verifing $u^i(t) >0$ for every $i=1,2,3$ and almost  every $t\in [0,T]$ is $i$-balanced for every $i=1,2,3$. Therefore, if  the control $u$ is optimal it satisfies (the usual maximum principle and) the  $i$-Legendre-Clebsch condition for all $i=1,2,3$ as well as the $(1,2)$-,the $(1,3)$- and the $(2,3)$-Goh conditions. 
		
		Instead, an optimal  control $\hat u=(\hat u^1,\hat u^2,\hat u^3):[0,T]\to U$ such that  $\hat u^1(t)\equiv 0$ while $\hat u^2(t) >0$ and $\hat u^3(t) >0$  for almost every $t\in [0,T] $, satisfies $i$-Legendre-Clebsch condition for $i=2,3$, and the $(2,3)$-Goh condition. 
	\end{example}

				\subsection{An worked out example of not  optimal control }\label{ESEMPIO}  Let us consider the optimal control problem 
				
				$$
				\left\{\begin{array}{l}	
					\qquad\qquad\qquad\qquad\qquad\min x_3(1),\\ \ds \frac{dx}{dt}= g_1(x(t))u^1(t) + g_2(x(t))u^2(t), \quad  a.e. \,\,\, t \in [0,1]
					\\ x(0) = { \footnotesize\begin{pmatrix}2\\0\\0\end{pmatrix}}\quad \qquad x(1) \in \mathfrak{T} \end{array}\right.$$
				where $$
				g_1(x):={ \small\begin{pmatrix}1\\0\\-x_2\end{pmatrix}} \quad g_2(x):={ \small\begin{pmatrix}0\\1\\0\end{pmatrix}}\qquad 	\mathfrak{T} := \{0\}\times \{0\} \times \rr,	
				$$	
				and the control set $U\subset \mathbb{Z}^2$ is defined as  $$ U : = \Big\{-4,-2,0,3\Big\} \times \Big\{-6,0,4,7\Big\}. $$
				We wish to establish whether the constant  control map
				$$\hat u(t) = \begin{pmatrix} -2\\\\0 \end{pmatrix}\qquad \forall t\in[0,1]$$
		
			is (feasible and) optimal. The trajectory $\hat x:[0,1]\to\rr^3$ corresponding to the control $\hat u$ is given by 
			$$
			\hat x(t) =\begin{pmatrix} 2-2t\\0\\0\end{pmatrix}\qquad t\in [0,1],
			$$
			which, in particular, says that  $\hat u$ is feasible, in that $\hat x(1) = 0\in \mathfrak{T} $. Let us check that the process $(\hat u,\hat x)$ verifies the first order Pontryagin Maximum Principle, namely there exists a pair $(\lambda,p)$ verifying ${\bf i)}-{\bf iv)}$ in Theorem \ref{TeoremaPrincipale}.   Indeed, since 
			$\hat x_2(t)=0 \,\,\forall t\in[0,1]$,
			the adjoint equation reduces to
			\bel{adj}
			\frac{dp}{dt}(t) = -p(t)\cdot \frac{\partial}{\partial x}\Big(g_1(x)\hat u^1(t) + g_2(x)\hat u^2(t) \Big)_{x=\hat x(t)} = (0,2p_3(t)\hat x_2(t)\hat u^1(t),0) =(0,0,0) \quad\forall t\in[0,1],
			\eeq
			while  the non-triviality and non-transversality condition read
			$$
			p({1})\in \Big\{(\alpha,\beta, -\lambda)\backslash \{0\},\,\,\,(\alpha,\beta)\in\rr^2, \lambda\geq 0\Big\}. \,\,
			$$
			Therefore
			one must have 
			$$
			p(t) = p(1)= (\alpha,\beta, -\lambda) \quad \forall t\in[0,1],
			$$
			for some $(\alpha,\beta)\in\rr^2$.
			First, let us see that, if 
			the maximization {\bf iv)} in Theorem \ref{TeoremaPrincipale} holds true, then   $\lambda$ cannot vanish. Indeed, such maximization implies that, for all   $\begin{pmatrix}u^1\\u^2\end{pmatrix}\in U$, one has
			\bel{max}
			\big(\alpha,\beta ,-\lambda\big)\cdot \Big(g_1(\hat x(t))(-2-u^1)+ g_2(\hat x(t))(0-u^2)\Big)\geq 0 \qquad \forall t\in[0,1]
			\eeq
			Taking $(u^1,u^2)=(-2,-6)$ and $(u^1,u^2)=(-2,4)$ in \eqref{max} we obtain
			$$
			6\big(\alpha,\beta,0\big)\cdot  g_2(\hat x(t))\geq 0, \qquad -4\big(\alpha,\beta,0\big)\cdot  g_2(\hat x(t))\geq 0,
			$$
			which implies $$0= \big(\alpha,\beta,\lambda\big)\cdot  g_2(\hat x(t)) = \big(\alpha,\beta,0\big)\cdot \begin{pmatrix}0\\1\\0\end{pmatrix}  = \beta $$
			On the other hand,  taking $(u^1,u^2)=(-4,0)$ and $(u^1,u^2)=(0,0)$ in \eqref{max} we get
			$$
			2\big(p_1(1),p_2(1),0\big)\cdot  g_1(\hat x(t))\geq 0, \qquad -2\big(p_1(1),p_2(1),0\big)\cdot  g_1(\hat x(t))\geq 0,
			$$
			which implies $$0= \big(\alpha,\beta ,0\big)\cdot  g_1(\hat x(t)) = \begin{pmatrix}1\\0\\\hat 0\end{pmatrix}  = \alpha.$$
			Therefore $p(t) = (\alpha,\beta,-\lambda )=(0,0,\-\lambda )$ for every $t\in [0,1]$. Hence we cannot choose  $\lambda=0$, for this choice would satisfy  the maximization relation {\bf iv)} in Theorem \ref{TeoremaPrincipale} but would contradict  the non-triviality condition {\bf i)}. Hence we must have 
			$$
			p(t) = (0,0,-\lambda) \qquad \text{with}\,\,\,\,\lambda \neq 0
			$$
			This implies that  the full higher order principle stated in  Theorem \ref{TeoremaPrincipale} is not satisfied, 
			in that the control $\hat u$ is Goh-fit a.e. while 
			$$
			p(t) \cdot [g_1,g_2](\hat x(t)) = (0,0,-\lambda) \cdot \begin{pmatrix}0\\0\\1 \end{pmatrix} = -\lambda\neq 0,
			$$
			namely the $(1,2)$-Goh condition {\bf v)} is violated. 
}

\section{  Variations}\label{Lie gen}
The proof of the main result (Theorem \ref{TeoremaPrincipale}) will be based on  set-separation arguments (see Section \ref{setsepsec}), which, in particular, require the generation  of adequate {\it variations} of the optimal process. As explained in the Introduction, this is the crucial issue  of the present paper, for we are going to  produce Lie bracket-based variations even when the classical assumption that the optimal control's values are 
in the interior of the control set $U$ (which might  be even empty) fails to be verified.

\subsection{Variation builders}

Let us introduce a class of indexes, that we shall call  the set of {\it variation signals}. 
\begin{definition}[Variation signals]\label{signals}
	Let us define the set $\mathfrak{V}$ of {\rm  variation signals} as the union
	$$\mathfrak{V}:=\mathfrak{V}_{ndl}\bigcup \mathfrak{V}_{Goh} \bigcup \mathfrak{V}_{LC2} \bigcup \mathfrak{V}_{LC3},$$
	where 
	$$\mathfrak{V}_{ndl}:=U,
	\,\,	\mathfrak{V}_{Goh}:=\{(i,j)\,\,\,i < j\,\,\,\, j,i=1,\ldots,m\},
	\,\,\mathfrak{V}_{LC2}:=\{(0,i)\,\,, i=1,\ldots,m\},
	\,\,\mathfrak{V}_{LC3}:=\{(1,0,1)\}.$$
	
\end{definition}

{{Every variation signal ${\bf c}\in\mathfrak{V}$ will represent  an associated  {\it variation builder}, in a way we now describe. 
		
		\begin{definition}[Variation builders]\begin{itemize}
			\item[{\bf i)}] If	${\bf c}\in\mathfrak{V}_{ndl}$, namely ${\bf c}=\bar u\in U$ 	the  associated variation builder  will be the standard needle variation corresponding to $\bar u$.  
			
			
			\item[{\bf ii)}] 	If ${\bf c}\in\mathfrak{V}_{Goh}$, that is ${\bf c}=(i,j)$, for some $0<i<j\leq m$, let us choose four positive real numbers  $\alpha^i,\beta^i$, $\alpha^j,\beta^j$    
			and let us define the times $\tau_0,\tau_1,\dots,\tau_4(=1/2)$  by setting
			$$\ds r:={(\alpha^i)}^{-1}+{(\alpha^j)}^{-1}+{(\beta^i)}^{-1}+{(\beta^j)}^{-1}$$
			$$\tau_0=0,\quad	\tau_1:=\ds \tau_{0}+\frac{({\alpha^i)}^{-1}}{2r}\quad 
			\tau_2:=\ds \tau_{1}+\frac{({\alpha^j})^{-1}}{2r}\quad
			\tau_3:=\ds \tau_{2}+\frac{({\beta^i})^{-1}}{2r}\quad
			\tau_4:=\ds \tau_{3}+\frac{({\beta^j})^{-1}}{2r} =\frac12
			$$
			
			The variation builder associated to ${\bf c}=(i,j)$ is  the (piece-wise constant) map $\gamma_{(i,j)}$ on $[0,1]$ 
			
			{
				\bel{leQij}
				\gamma_{(i,j)}(s) :=\left\{\begin{array}{ll} 	\,\,\,\,\hat\gamma_{(i,j)}(s)\quad&\forall s\in[0,1/2]\\
					-\hat\gamma_{(i,j)}(s-1/2)\quad&\forall s\in[1/2,1]				
				\end{array}           \right.
				\eeq
				where $\hat\gamma_{(i,j)}$ is defined (on $[0,1/2])$ as follows:
				$$
				\hat\gamma_{(i,j)}(s) := \alpha^i\mathbf{1}_{\left[0,\tau_1\right]}(s)\mathbf{e}_i +\alpha^j\mathbf{1}_{\left[\tau_1,\tau_2\right] }(s)\mathbf{e}_j -\beta^i\mathbf{1}_{\left[\tau_2,\tau_3\right]}(s)\mathbf{e}_i -\beta^j\mathbf{1}_{\left[\tau_3,1\right] }(s)\mathbf{e}_j \quad\forall s\in [0,1/2]
				$$
				More explicitly, if we consider the times 
				$$ \tau_5:=\ds\tau_{4}+\frac{({\beta^j})^{-1}}{2r}\quad \tau_6:=\ds \tau_{5} +\frac{({\beta^i})^{-1}}{2r}\quad
				\tau_7:=\ds \tau_{6} +\frac{({\alpha^j})^{-1}}{2r}\quad \tau_8:=\ds \tau_{7} + \frac{({\alpha^i})^{-1}}{2r} =1 ,	
				$$	
				we get, for every $s\in[0,1]$, 		\bel{leQij}\begin{array}{l}\ds
					\gamma_{(i,j)}(s) := \alpha^i\mathbf{1}_{\left[0,\tau_1\right]}(s)\mathbf{e}_i +\alpha^j\mathbf{1}_{\left[\tau_1,\tau_2\right] }(s)\mathbf{e}_j -\beta^i\mathbf{1}_{\left[\tau_2,\tau_3\right]}(s)\mathbf{e}_i -\beta^j\mathbf{1}_{\left[\tau_3,1\right] }(s)\mathbf{e}_j \\\\
					\qquad- \beta^i\mathbf{1}_{\left[\tau4,\tau_5\right]}(s)\mathbf{e}_i -\beta^j\mathbf{1}_{\left[\tau_5,\tau_6\right] }(s)\mathbf{e}_j +\alpha^i\mathbf{1}_{\left[\tau_6,\tau_7\right]}(s)\mathbf{e}_i+\alpha^j\mathbf{1}_{\left[\tau_7,1\right] }(s)\mathbf{e}_j .\end{array}
				\eeq}
			
			(Notice that $\gamma_{(i,j)}([0,1])\subset \mathbf{e}_i\rr\times \mathbf{e}_j\rr$).
			
			\item[{\bf iii)}] If we consider a variation signal  ${\bf c}\in \mathfrak{V}_{LC2}$, namely ${\bf c}=(0,i)$ for some  $i\in\{1,\dots m\}$, after  choosing  two positive numbers  $\alpha^i,\beta^i$ and setting $\ds\bar\tau:= \frac{\alpha^i}{\alpha^i+\beta^i}$, let us associate the  variation builder  
			{\bel{leQi}
			\gamma_{(0,i)}(s) := \gamma_{(0,i)}^i(s)\mathbf{e}_i
			:= \Big(-\beta^i\mathbf{1}_{\left[0,\bar\tau\right]}(s) +\alpha^i\mathbf{1}_{\left[\bar\tau,1\right] }(s)\Big)\mathbf{e}_i \qquad \forall s\in [0,1]
			\eeq
			
			}
			
			\item[{\bf iv)}] If ${\bf c}\in \mathfrak{V}_{LC3}$, namely  $m=1$ and ${\bf c}=(1,0,1)$ ,
			after  choosing  two positive numbers  $\alpha^1,\beta^1$ and setting $\ds{\hat\tau}_1:=\frac{\alpha^1}{2(\alpha^1+\beta^1)} \quad {\hat\tau}_2:=1-\hat\tau_1=\frac{\alpha^1+2\beta^1}{2(\alpha^1+\beta^1)}$, as variation builder   let us consider the map 
		{ $$
			\gamma_{(1,0,1)}(s) := \gamma_{(1,0,1)}^1(s) \mathbf{e}_1 := \Big(-
			\beta^1\mathbf{1}_{\left[0,{\hat\tau}_1\right]}(s)\mathbf +\alpha^1\mathbf{1}_{\left[{\hat\tau}_1 ,{\hat\tau}_2\right] } - \beta^1\mathbf{1}_{\left[{\hat\tau}_2,1\right] }\Big) \mathbf{e}_1 \qquad\forall s\in [0,1]$$
			}
				\end{itemize}	\end{definition}
		\subsubsection{Primitives of the variation generators $\gamma_{i,j},\gamma_{0,i}$ $\gamma_{1,1,0}$} To prove Theorem \ref{eccolestime} below, we will consider peculiar properties of  the  (continuous, piecewise linear)  primitives 
		{$$ \Gamma_{(i,j)}(s):=\int_0^s \gamma_{(i,j)}(\sigma)d\sigma\qquad   \Gamma_{(0,i)}(s): =\int_0^s \gamma_{(0,i)}(\sigma)d\sigma,\quad s\in [0,1],\,\,\, 1\leq i<j\leq m $$
			$$
			\Gamma_{(1,0,1)}(s) := \ds\int_0^s \gamma_{(1,0,1)}(\sigma)d\sigma, \qquad \quad \forall s\in [0,1]. $$

			Notice that, for every $s\in [0,\tau_4]  = [0,1/2]$, one has
			$$ \Gamma_{(i,j)}(s) = \hat\Gamma_{(i,j)}(s) := \int_0^s \hat\gamma_{(i,j)}(\sigma)(\sigma) = \left\{\begin{array}{ll}s\alpha^i\mathbf{e}_i\quad &\forall s\in \left[0,\tau_1\right]\\
				(2r)^{-1}\mathbf{e}_i	
				+(s-\tau_1)\alpha^j\mathbf{e}_j\quad &\forall s\in \left[\tau_1,\tau_2\right] \\
				(2r)^{-1}\mathbf{e}_i	+ (2r)^{-1}\mathbf{e}_j	  -(s-\tau_2)\beta^1\mathbf{e}_i\quad &\forall s\in \left[\tau_2,\tau_3\right]\\ (2r)^{-1}\mathbf{e}_j -(s-\tau_3)\beta^j\mathbf{e}_j &\forall s\in \left[\tau_3,1/2\right].
			\end{array}\right.
			$$
			Therefore, for every $s\in [0,\tau_8]=[0,1]$ one gets 	
			\bel{Gammaij}
			\Gamma_{(i,j)}(s)=
			\left\{\begin{array}{ll}\ds \hat\Gamma_{(i,j)}(s)  \quad & \forall s\in [0,1/2]
				\\ \ds  \hat\Gamma_{(i,j)}(1/2)  - \hat\Gamma_{(i,j)}(s-1/2)  & \forall s\in [1/2,1]. \end{array} \right.
			\eeq
			Indeed,
			$$\begin{array}{c} \ds\Gamma_{(i,j)}(s)=
				\int_0^s \left(\gamma_{(i,j)}(\sigma){\bf 1}_{[0,1/2]}(\sigma) +  \gamma_{(i,j)}(\sigma){\bf 1}_{[1/2,1]}(\sigma)\right) d\sigma \\\\
				\ds	=\int_0^s \Big(\gamma_{(i,j)}(\sigma){\bf 1}_{[0,1/2]}(\sigma) -  \hat\gamma_{(i,j)}(\sigma-1/2){\bf 1}_{[1/2,1]}(\sigma)\Big) d\sigma
				\\\\ =\ds\left\{\begin{array}{ll}\ds \hat\Gamma_{(i,j)}(s) \quad & \forall s\in [0,1/2]
					\\ \ds  \hat\Gamma_{(i,j)}(1/2)  - \int_{1/2}^s \hat\gamma_{(i,j)}(\sigma-1/2)d\sigma & \forall s\in [1/2,1] \end{array} =\right. 
				\end{array} 
			$$
			$$
	\begin{array}{c} = \ds\left\{\begin{array}{ll}\ds \hat\Gamma_{(i,j)}(s)  \quad & \forall s\in [0,1/2]
					\\ \ds  \hat\Gamma_{(i,j)}(1/2)  - \int_{0}^{s-1/2} \hat\gamma_{(i,j)}(\xi)d\xi & \forall s\in [1/2,1] \end{array} \right. =
				\left\{\begin{array}{ll}\ds \hat\Gamma_{(i,j)}(s)  \quad & \forall s\in [0,1/2]
					\\ \ds  \hat\Gamma_{(i,j)}(1/2)  - \hat\Gamma_{(i,j)}(s-1/2)  & \forall s\in [1/2,1] \end{array} \right.
			\end{array}
			$$
			
			More explicitly, one has 	
			
			$$ \Gamma_{(i,j)}(s):=\int_0^s \gamma_{(i,j)}(\sigma)d\sigma = \left\{\begin{array}{ll}s\alpha^i\mathbf{e}_i\quad &\forall s\in \left[0,\tau_1\right]\\
				(2r)^{-1}\mathbf{e}_i	
				+(s-\tau_1)\alpha^j\mathbf{e}_j\quad &\forall s\in \left[\tau_1,\tau_2\right] \\
				(2r)^{-1}\mathbf{e}_i	+ (2r)^{-1}\mathbf{e}_j	  -(s-\tau_2)\beta^i\mathbf{e}_i\quad &\forall s\in \left[\tau_2,\tau_3\right]\\ (2r)^{-1}\mathbf{e}_j -(s-\tau_3)\beta^j\mathbf{e}_j &\forall s\in \left[\tau_3,\tau_4\right]
				\\ -(s-\tau_4)\alpha^i\mathbf{e}_i &\forall s\in \left[\tau_4,\tau_5\right]
				\\ -(2r)^{-1}\mathbf{e}_i -(s-\tau_5)\alpha^j\mathbf{e}_j &\forall s\in \left[\tau_5,\tau_6\right]
				\\ -(2r)^{-1}\mathbf{e}_i -(2r)^{-1}\mathbf{e}_j +(s-\tau_6)\beta^i\mathbf{e}_i &\forall s\in \left[\tau_6,\tau_7\right]
				\\ -(2r)^{-1}\mathbf{e}_j +  (s-\tau_7)\beta^j\mathbf{e}_j   &\forall s\in \left[\tau_7,1\right]
			\end{array}\right.
			$$}
		
		Furthermore, for every $i=1\ldots,m$ one gets
		{
		$$  \Gamma_{(0,i)}(s): = \Gamma_{(0,i)}^i(s) \mathbf{e}_i=\int_0^s \gamma_{(0,i)}(\sigma)d\sigma =
		\left\{\begin{array}{ll}  -s\beta^i\mathbf{e}_i\quad &\forall s \in \left[0,\bar\tau\right]\\\\ \big(-\bar\tau\beta^1+(s-\bar\tau)\alpha^i\big)\mathbf{e}_1\ \quad&\forall s \in \left[\bar\tau,1\right]\end{array}\right. $$}
		{and, finally, 
			{$$\begin{array}{l}
			\Gamma_{(1,0,1)}(s):=\Gamma_{(1,0,1)}^1(s) \mathbf{e}_1=\ds\int_0^s \gamma_{(1,0,1)}(\sigma)d\sigma =\\\qquad \qquad\qquad\qquad\qquad\qquad\qquad\qquad =\left\{\begin{array}{ll} -s\beta^1\mathbf{e}_1\quad&\forall s\in\left[0,{\hat\tau}_1\right]\\\\ \Big(-{\hat\tau}_1\beta^1+(s-{\hat\tau}_1)\alpha^1\Big)\mathbf{e}_1&\forall s\in\left[{\hat\tau}_1,{\hat\tau}_2\right]\\\Big(-{\hat\tau}_1\beta^1 +
				({\hat\tau}_2-{\hat\tau}_1)\alpha^1-(s-{\hat\tau}_2)\beta^1\Big)\mathbf{e}_1\,\, &\forall s\in\left[{\hat\tau}_2,1\right],
			\end{array}\right.\end{array}$$}}
		For every $h=0,\ldots,m$  let us use  $\Gamma_{(i,j)}^h$, $\Gamma_{(0,i)}^h$,
		and $\Gamma_{(1,0,1)}^h$  to denote the $h-th$ components  of  $\Gamma_{(i,j)}$  $\Gamma_{(0,i)}$, and $\Gamma_{(1,0,1)}$, respectively. Notice that, in particular, 
		{  that  $\Gamma_{(i,j)}^0(s)=\Gamma_{(0,i)}^0(s) = \Gamma_{(1,0,1)}^0(s)  =0$, $\forall s\in[0,1] .$ }

		The following two lemmas will be crucial in the construction of control variations.
		\begin{lemma}\label{integrali}
			For every 
			$i,j\in \{1,\ldots,m\}$, $i<j$, one has
			{ \bel{periodic}
			\begin{array}{l}\ds	0=\Gamma_{(i,j)}(0)=\Gamma_{(i,j)}(1)=\Gamma_{(0,i)}(0)=\Gamma_{(0,i)}(1)\\\\ \ds0=\Gamma_{(1,0,1)}^1(0)=\Gamma_{(1,0,1)}^1(1)\quad  \quad \int_0^1 \Gamma_{(1,0,1)}(s)ds = 0,
			\end{array} 
			\eeq  }
			\bel{int0}
			\bigintsss_0^1 \Gamma_{(i,j)(s)}ds
			= 0,
			\eeq	and
			\bel{oi}\bigintssss_0^1\Gamma_{(0,i)}^h(s)ds = \left\{\begin{array}{ll}\ds -\frac{\alpha^i \beta^i}{2(\alpha^i+\beta^i)}\quad &\text{if}\,\,\, h=i\\\\
				\ds	\,\,\,\,\,\,\,\,\,\,\,\,\,0 &\text{if}\,\,\, h \in\{1,\dotsm\}\backslash\{i\} \end{array}\right.
			\eeq
		\end{lemma}
		\begin{proof} Equalities \eqref{periodic} are trivial.  As for \eqref{int0}, by \eqref{Gammaij}  one gets 
			$$
			\bigintsss_0^1 \Gamma_{(i,j)}(s)ds = \bigintsss_0^{1/2} \hat\Gamma_{(i,j)}(s)ds  - \bigintsss_{1/2}^1 \hat\Gamma_{(i,j)}(s-1/2)ds =  \bigintsss_0^{1/2} \hat\Gamma_{(i,j)}(s)ds -  \bigintsss_0^{1/2} \hat\Gamma_{(i,j)}(s)ds=0		
			$$
			Equality \eqref{oi} is just the computation of the integral of a piecewise linear continuous map on $[0,1]$ . 
		\end{proof}
		
		Other important quantities will be the integrals of  products like $\Gamma_{(i,j)}^i\gamma_{(i,j)}^j(t)$, $\Gamma_{(i,0)}^i\gamma_{(i,0)}^0$,\linebreak and $\Gamma_{(1,0,1)}^1\gamma_{(1,0,1)}^j(t)$, which are connected to certain areas:
		\begin{lemma}\label{integrali2} For every pair $(i,j)$, $1\leq i<j\leq m$ one has 
			\bel{areazero1}
			\int_0^1 \Gamma_{(i,j)}^i(t)\dot \Gamma_{(i,j)}^j(t)dt =  Area\Big(\Gamma_{(i,j)}^i,\Gamma_{(i,j)}^j\Big) =\frac{1}{2r^2} \qquad\footnotemark\eeq
			\footnotetext{Because of the inverted direction, one obviously has $ Area\Big(\Gamma_{(i,j)}^j,\Gamma_{(i,j)}^i\Big)=-\frac{1}{2r^2}$}
			Moreover, for all $ h \in\{0,\dots m\}\backslash\{i,j\}$,
			{\bel{areazero2} \int_0^1 \Gamma_{(i,j)}^i(t)\dot\Gamma_{(i,j)}^h(t)dt = Area\Big(\Gamma_{(i,j)}^i,\Gamma_{(i,j)}^h\Big) =0 = Area\Big(\Gamma_{(i,j)}^j, \Gamma_{(i,j)}^h\Big)=\int_0^1 \Gamma_{(i,j)}^j(t)\dot\Gamma_{(i,j)}^h(t)dt \eeq}

		\end{lemma} 
		\footnotetext{The introduction of the functions ${\tilde\Gamma}_{(0,i)}^h$, ${\tilde\Gamma}_{(1,0,1)}^0$, ${\tilde\Gamma}_{(1,0,1)}^1$ is motivated by the fact that, in order to consider the notion of $Area$, we need {\it closed} curves. }

		\begin{proof}	
			Equality \eqref{areazero1}, can obviously obtained by direct computation. However it  is geometrically  deductible  
			by the fact that the  image $s\mapsto (\Gamma_{i,j}^i, \Gamma_{i,j}^i) ([0,1])$  is the the union of two equal squares --run in counterclockwise sense-- whose  edges' lenghts are  $1/2r$. So  
			$\ds Area\Big(\Gamma_{(i,j)}^i,\Gamma_{(i,j)}^j\Big) = 2\cdot (1/2r)^2 = \frac{1}{2r^2}.$
			
			Equalities \eqref{areazero2},\eqref{areazero3}, \eqref{areazero4},  and  \eqref{areazero5} are trivially  deduced by computation.
	\footnote{Incidentally,  equality  \eqref{areazero3} can also be  interpreted it geometrically, so obtaining   $ Area\Big({\tilde\Gamma}_{(0,i)}^i,{\tilde\Gamma}_{(0,i)}^0\Big) = - \int_0^1\Gamma_{(0,i)}^h(s)ds$ (where the minus sign comes from opposite direction of rotations).}\quad\footnote{One can argue geometrically also  for the first equality in \eqref{areazero5}, in that  it reduces to determining the signed area of the trapezoid spanned by $\Gamma_{(1,0,1)}^1$, which comprises two triangles having opposite areas.  }	\end{proof}
		
		%
		%

\subsection{Control variations}
\begin{definition}[Families of control variations]\label{defvar}
	Let us fix a map  $u\in L^\infty([0,T],\rr^m) $, 
	a Lebesgue point $\bar t\in]0,T[$ for $u$,\footnote{Let us recall that, if $\varphi:[a,b]\mapsto \rr^k$ is a $L^1$ map, a $t\in ]a,b[$ is called a Lebesgue point for $\varphi$ if  $$\ds\lim_{r\to 0+} \frac{1}{2r}\int_{t-r}^{t+r}|\varphi(s)-\varphi(t)|ds =0.  $$} and a parameter $\varepsilon>0$ such that $\sqrt[3]{\varepsilon}\le\ol{t}$.
	\begin{itemize}
		\item If  $\mathbf{c}=\hat u\in \mathfrak{V}_{ndl}(=U)$  we define   the family  of $\varepsilon$-dependent  controls ${\big\{ {{ u}}_{\varepsilon,\mathbf{c},\ol{t}},  \,\,\, \varepsilon \in [0,\ol{t})\big\}}$ by setting \bel{perneedle}{{u}}_{\varepsilon,\mathbf{c},\ol{t}}(t):=\begin{cases}{{u}}(t) &\text{ if } t \in [0,\ol{t}-\varepsilon) \cup (\ol{t},T]\\  \hat u &\text{ if } t \in [\ol{t}-\varepsilon,\ol{t}],\end{cases}\eeq 
		This  family  is usually referred as a {\rm needle variation} of $u$ at $\bar t$. 	
		\item 
		If
		$\mathbf{c}=(i,j)\in  \mathfrak{V}_{Goh} $,  $1\leq i<j\leq m$, and the control $u$ is both $i$-balanced and $j$-balanced  at $\bar t $,   we define the $\varepsilon$-parameterized family   ${\Big\{ {{u}}_{\varepsilon,\mathbf{c},\ol{t}},\,\,\, 0<\sqrt{\varepsilon}\le\ol{t}\Big\}}$ of controls by setting
		\bel{perGoh}{u}_{\varepsilon, \mathbf{c},\ol{t}}(t):=\begin{cases}
			{{u}}(t) & \text{ if } t \not \in [\ol{t}-\sqrt{\varepsilon},\ol{t}]\\\ds
			{{u}}(t)+ \gamma_{(i,j)}\left(\frac{t-(\ol{t}-\sqrt{\varepsilon})}{\sqrt{\varepsilon}}\right)& \text{ if }t \in [\ol{t}-\sqrt{\varepsilon},\ol{t}] .
		\end{cases}\eeq
		\item 
		If
		$\mathbf{c}=(0,i)\in \mathfrak{V}_{LC2} $,  $0<i\leq m$, and the control $u$ is  { $i$-balanced  at $\bar t $},   we define the $\varepsilon$-parameterized family   ${\Big\{ {{u}}_{\varepsilon,\mathbf{c},\ol{t}}(t), \,\,\, 0<\sqrt{\varepsilon}\le\ol{t},\, \Big\}}$ of controls by setting
		\bel{perLC2}{u}_{\varepsilon, \mathbf{c},\ol{t}}(t):=\begin{cases}
			{{u}}(t) & \text{ if } t \not \in [\ol{t}-\sqrt{\varepsilon},\ol{t}]\\\ds
			{{u}}(t)+ \gamma_{(0,i)}\left(\frac{t-(\ol{t}-\sqrt{\varepsilon})}{\sqrt{\varepsilon}}\right)& \text{ if }t \in [\ol{t}-\sqrt{\varepsilon},\ol{t}] .
		\end{cases}\eeq
		\item 
		If $m=1$, $g:=g_1$, 
		$\{\mathbf{c} =(1,0,1)\} = \mathfrak{V}_{LC3}$, and $u$ is $1$-balanced  at $\bar t$,
		we define the family \linebreak ${\Big\{ {{u}}_{\varepsilon,\mathbf{c},\ol{t}}(t),\,\,\, \, 0<\sqrt[3]{\varepsilon}\le\ol{t}\Big\}}$ of controls by setting 
		\bel{perLC}{{u}}_{\varepsilon, \mathbf{c},\ol{t}}(t):=\begin{cases}
			{{u}}(t) & \text{ if } t \not\in [\ol{t}-\sqrt[3]{\varepsilon},\ol{t}]\\\ds
			{{u}}(t)+ \gamma_{(1,0,1)}\left(\frac{t-(\ol{t}-\sqrt[3]{\varepsilon})}{\sqrt[3]{\varepsilon}}\right)& \text{ if }t \in [\ol{t}-\sqrt[3]{\varepsilon},\ol{t}]
		\end{cases}\eeq
\end{itemize}\end{definition}

\begin{remark}Let us observe  that, for  every  control $u$ and every variation signal $ \mathbf{c}\in \mathfrak{V}$ as in Definition \ref{defvar}, the corresponding perturbed control ${{u}}_{\varepsilon, \mathbf{c},\ol{t}}$  is admissible for every $\varepsilon>0$, i.e. it takes values in $U$.    \end{remark}

Part {\bf 2)}-{\bf 4)}  of following theorem establish how the above control variations produce variations of the trajectories, in terms of Lie brackets (while part {\bf 1)} deals with the standard needle variations.).
\begin{theorem}  \label{eccolestime} Let  $(u,x)$ be a process for the control system in {\rm (1)}, and let $\ol{t}\in]0,T]$ be a  Lebesgue point
	for $u$ such that  the vector fields  $f$ and $g_i$  are  of class $C^1$ around $x(\ol{t})$. 
	Then, for any
	any  $\mathbf{c} \in \mathfrak{V} $ and any $\varepsilon>0$ sufficiently small, using $x_\varepsilon$ to denote the trajectory corresponding to the perturbed control ${u}_{\varepsilon,\mathbf{c},\ol{t}}$ (see Definition \ref{defvar}), 
	the following four statements hold true.
	
	
	\item[{\bf 1)}] If $\bar u \in \mathfrak{V}_{ndl}$,
	then  \begin{equation}\label{ndl}
		x_\varepsilon(\ol{t})-x(\ol{t})=\varepsilon\sum_{\ell=1}^mg_\ell(x(\ol{t}))\big(\bar u^\ell - u^\ell(\ol{t})\big)+o(\varepsilon)\,\,\end{equation}
	\item[{\bf 2)}] If $\mathbf{c}=(i,j) \in \mathfrak{V}_{Goh}$, , for some $1\leq i<j\leq m$,  and  the control $u$ is  both $i$-balanced and $j$-balanced at $\bar t $,
	then there exists $M>0$  such that \begin{equation}\label{2}
		x_\varepsilon(\ol{t})-x(\ol{t})=M\varepsilon[g_i,g_j](x(\ol{t}))+o(\varepsilon).\,\,\end{equation}
	
	\item[{\bf 3)}] If  $\mathbf{c}=(0,i) \in \mathfrak{V}_{LC2}$, for some $i=1,\ldots,m$, and  the control $u$ is  { \rm $i$-Legendre\textendash{}Clebsch-fit   at $\bar t $},
	then there exists $\tilde M>0$  such that
	\begin{equation}\label{1} x_\varepsilon(\ol{t})-x(\ol{t})=\tilde M\varepsilon[f,g_i](x(\ol{t}))+o(\varepsilon)\,\,
	\end{equation}
	\item[{\bf 4)}] If $m=1$, $\mathbf{c} = \mathfrak{V}_{LC3}$,    the control $u$ is  { \rm Legendre\textendash{}Clebsch-fit  of step $3$ at $\bar t $}, and $f,g (:=g_1)$ are of class $C^2$ around ${x(\ol{t})}$,
	then there exists $\check M>0$ such that \begin{equation}\label{3} x_\varepsilon(\ol{t})-x(\ol{t})=\check M\varepsilon[g,[f,g]](x(\ol{t}))+o(\varepsilon)\,\,\end{equation}	
\end{theorem}
Tee proof of this theorem will be given in Section \ref{proof3.6}
\vskip 0.7truecm
The following remark will be crucial to derive equalities {\bf v)} and {\bf vi)}  in Theorem \ref{TeoremaPrincipale} from analogous inequalities.

\begin{remark}\label{imprem} It is trivial to check that, if in the previous constructions  we replaced the maps $\Gamma_{(i,j)}$ and 
$\Gamma_{(0,i)}$, $1\leq i<j\leq m$ with the maps  $\check\Gamma_{(i,j)}$ and $\check\Gamma_{(0,i)}$ defined as 
$$
\check\Gamma_{(i,j)}(s):= \Gamma_{(i,j)}(1-s) \quad\text{and}  \quad \check\Gamma_{(0,i)}:= \Gamma_{(0,i)}(1-s) \qquad \forall s\in [0,1],
$$ 
respectively,  we would obtain
\begin{equation}\label{2-}
	x_\varepsilon(\ol{t})-x(\ol{t})=-M\varepsilon[g_i,g_j](x(\ol{t}))+o(\varepsilon).\,\,\end{equation}
in place of \eqref{2} and 
\begin{equation}\label{1-} x_\varepsilon(\ol{t})-x(\ol{t})=-\tilde M\varepsilon[f,g_i](x(\ol{t}))+o(\varepsilon),\,\,
\end{equation}
 in place of \eqref{1}. \end{remark}

{ 	 \section{Approximation of solutions}\label{approx.sect} In this  section  we will recall some   classical facts (see e.g. \cite{Schattler}, \cite{chen}, \cite{SussmannAGDQ}, \cite{Sussmann2})  about local  approximations of trajectories of control affine systems through product of exponential maps of Lie brackets up to a certain degree. Together with the use of variations builders presented in the previous sections, these kinds of approximation will allow us to prove Theorem \ref{eccolestime}.
%

	\subsection{The case with sub-linear controls}
	Most of the results are of local nature, so we limit our exposition to Euclidean spaces $\rr^n$, $n\geq 1$. Let us begin with a trivial remark:   if  $X$ is a given Lipschitz continuous vector field defined on an open subset $\Omega\subseteq\rr^n$ and $b\in L^1([0,T], \rr)$, if, for some $T>0$, a (necessarily unique) solution
	$[0,T]\ni t\mapsto x(t)$  of the  Cauchy problem {  $$\dot x(t) = b(t) X(x(t))\,\,\,	x(0)=\hat{x}\quad t\in [0,T]$$
	exists, then 
		$$  x(t)= e^{B(t)X}(\hat{x})\quad \forall t\in [0,T], \footnotemark$$}	
	{	where  we have set $B(t) = \ds \int_0^t b(s)ds, \,\forall t\in[0,T]$.  In other words, the value $x(t)$ of the solution at time $t$ coincides with  the value $\tilde x(\tau)$  of the solution of the autonomous problem $\dot{\tilde x}(\tau) =  X(\tilde x(\tau))$, $\tilde x(0)=\hat x$ at time $\tau = B(t)$. } 
	\vskip12truemm
	An estimate of the error caused by  non-commutativity is given by the following elementary  result:
	\begin{lem}\label{lemmasemplificativo}
		Let $m$ be any natural number, let $\chi>0$ be a positive constant and let $b_i(s):[0,T]\to \rr, \quad i=0,\ldots,m$ be $m+1$ arbitrary  $L^1$ functions such that $|b_i(s)|\le \chi s$,  
		 for almost every $s\in [0,T]$. Let  us consider $m+1$ vector fields  $X_i \in C^1(\Omega;\rr^n)$ defined on an open subset $\Omega\subseteq \rr^n,\quad i=0,\ldots,m$, and  let $x(\cdot)$ be a solution on an interval $[0,T]$ of the  differential equation
			{ \begin{equation}\label{eqdiffcomplicatalemma}	\dot{x}(t)=\sum\limits_{i=0}^m b^i(t) \, X_i\left(x(t)\right)\end{equation} }
		{
			Then, setting,  for every $t\in (0,T)$ and every $i=0,\ldots,m$, $\ds B_i(t):=\int_0^t b^i(s) ,\ds$ we have 
		{\begin{equation}\label{primo}
				x(t)=
				e^{B^{0}(t)X_0}\circ\ldots\circ e^{B^{m}(t)X_m}(x(0)) +o(t^3) \quad t\in [0,T],
		\end{equation}}}
	
	\end{lem}
	For a proof we refer the reader to the vast literature on this subject (see e.g.\cite{kaw suss}, \cite{Schattler}, \cite{arxiv}). We just recall that, in order to prove the above resultan essential use of this result is made  of  the following  basic fact:
	\begin{lem}\label{sviluppi} If $\Omega\subseteq \rr^q$ (for some $q\geq 1$) is an open subset, 
		 {  $X,Y \in C^n(\Omega;\rr^q)$ for some $n\geq 1$}, and $x\in\Omega$,  then
		{	\begin{equation}\label{derivateordinealto}
				\frac{d^n}{d\tau^n} \Big(D(e^{-\tau X})\cdot Y(e^{\tau X}x)\Big) = D(e^{-\tau X})\cdot\ad^n X(Y)(e^{\tau X}x), \qquad \text{ for all } n\ge 1,
		\end{equation} }
	for any sufficiently small $\tau$,
		where we have used the ''ad'' operator defined recursively  as  
		{$$\ad^0 X(Y)=Y\qquad \ldots\qquad \ad^n X(Y)=[X,\ad^{n-1}X(Y)],\,\,\,\forall n\in \mathbb{N}.$$}
		In particular,  
			 we get the Taylor expansion  { 	\begin{equation} \label{Taylor}
				D(e^{-\tau X})\cdot\left(Y(e^{\tau X}x)\right)=\sum\limits_{i=0}^n \frac{\tau^i}{i!}\ad^i X(Y) (x) + o(\tau^n)
		\end{equation}}

	\end{lem}

	In Proposition \ref{ispiratoaSchattler1} below the hypothesis of sublinearity of the controls is dismissed, and  an  estimate is given  for  the flow of $\dot x =\sum\limits_{i=0}^m a^i(s)\cdot g_i(x(s))$   in terms of the  the original vector fields $g_0,\ldots,g_m$ and their  Lie brackets.
	\begin{prop}\label{ispiratoaSchattler1}  For any $i=0,\ldots,m$, consider a function $a^i(s)\in L^\infty([0,T];\rr)$ and 
		let $g_i$ be a   vector field belonging to $C^2(\Omega;\rr^n)$.
		 Then, if we set, for all $ i=0,\ldots,m$ and 
		$i<j\leq m$,
		\bel{A}
		\begin{array}{c}A^i(s):=\ds\int_0^s a^i(\sigma)d\sigma \quad A^{i,j}(s):=\ds\int_0^s A^i(\sigma)\frac{dA^j}{d\sigma}(\sigma)d\sigma  = \int_0^s A^i(\sigma)a^j(\sigma)d\sigma  \qquad s\in [0,T],\end{array}
		\eeq
		every solution $x(\cdot ) $ 	{\rm(}on some interval $[0,\tau]${\rm)}
		to the ordinary differential equation
		\begin{equation}\dot {x}(t)=\sum\limits_{i=0}^m a^i(t)\,\,g_i (x(t))\end{equation}
	 verifies 
		{ 
			\begin{equation}\label{Cruciale1standard}\begin{array}{c}
					x(t)=
					e^{A^{0}(t)g_{0}}\circ\ldots\circ e^{A^{m}(t)g_{m}}\circ e^{A^{0,1}(t)[g_0,g_1]}\circ\dots\circ e^{A^{m-1,m}(t)[g_{m-1},g_m])}(x(0)) + o(t^2)\end{array} 
		\end{equation}}
		(where it is meant that the order in the product with two indexes is the lexicographic one.)
	\end{prop}

	As a straightforward consequence of Proposition \ref{ispiratoaSchattler1},  one gets
	\begin{cor}\label{corollariodischattlerutile}Let us fix $t>0$ (sufficiently small), let us choose $i^*,j^*\in \{0,\ldots,m\}$, $j^*>i^*$ and let $g_{j^*},g_{i^*}\in C^2(\Omega,\rr^n)$. If, for every $i=0,1,\ldots,m$, the $L^\infty$ controls  $a^i:[0,t]\to U $ in Proposition \ref{ispiratoaSchattler1} are chosen   such that $A^i(t)=0$ and $A^{j^*,i^*}(t)=ct^2+o(t^2)$ with $c>0$ is the only non-vanishing term in the family $\{A^{ji}(t)\}_{j>i,{  i=0,\dots,m}}$, then one has $$x(t)-x(0)=ct^2[g_{i^*},g_{j^*}](x(0))+o(t^2).$$
	\end{cor}

	Finally, we provide an  estimate   for  the flow of our differential equation  $\dot {x}(t)=\sum\limits_{i=0}^m a^i(t)\,\,g_i (x(t))$   in terms of the  Lie brackets up to the length $3$:

	{ \begin{prop} \label{ispiratoaSchattler3}
			Let $a_i(s), \quad i=0,\ldots,m$ be $m+1$ arbitrary functions in $L^\infty([0,T];\rr)$.\\
			Let $t\in (0,T)$ and let $g_i$,  $i=0,\ldots,m$, be $m+1$ vector fields of belonging to  $C^3(\Omega,\rr^n)$.  Then, if we set, for all $k,i,j$ with  $ i=0,\ldots,m$,
			$k,i<j\leq m$,
			\bel{A}
			\begin{array}{c}A^i(s):=\ds\int_0^s a^i(\sigma)d\sigma \quad A^{i,j}(s):=\ds\int_0^s A^i(\sigma)\frac{dA^j}{d\sigma}(\sigma)d\sigma  = \int_0^s A^i(\sigma)a^j(\sigma)d\sigma\\\\
				A^{kij}(s) := \ds\int_0^s A^k(\sigma)\frac{dA^{i,j}}{d\sigma}(\sigma)d\sigma  = \int_0^s A^k(\sigma) A^i(\sigma)a^j(\sigma)d\sigma
			\end{array}  \qquad s\in [0,T],
			\eeq
			every solution $  x(\cdot ) ${\rm(}on some interval $[0,\tau]${\rm)}
			to the ordinary differential equation
		$\ds \dot {x}(s)=\sum\limits_{i=0}^m a^i(s)\,\,g_i (x(s))$
			verifies 
			{	 \begin{equation}\label{CrucialeSt}\begin{array}{ll}x(t)&=e^{A^{0}(t)g_{0}}\circ\ldots\circ e^{A^{m}(t)g_m}\circ \\\\
						& e^{A^{0,1}(t)[g_0,g_1]}\circ\dots\circ e^{A^{m-1,m}(t)[g_{m-1},g_m]}\circ\\\\& e^{A^{0,0,1}(t)[g_0,[g_0,g_1]]}\circ e^{A^{0,0,1}(t)[g_0,[g_0,g_2]]}\circ\ldots \circ e^{A^{m-1,m-1,m}(t)[g_{m-1},[g_{m-1},g_m]]}\circ\\\\ & e^{A^{1,0,1}(t)[g_1,[g_0,g_1]]}\circ e^{A^{1,0,2}(t)[g_1,[g_0,g_2]]}\circ\ldots
						\circ  e^{A^{m,m-1,m}(t)[g_m,[g_{m-1},g_m]]}\big(x(0)\big)
						+ o(t^3)\end{array}\end{equation}}
		\end{prop}

		\begin{remark}\label{espansione3} We will use this third order  approximation only when $m=1$, so that, if one sets $f:=g_0$, $g:=g_1$, it reduces to 
				\begin{equation}\label{CrucialeGiusta?}x(t)=	e^{A^{0}(t)f}\circ	e^{A^{1}(t)g}\circ e^{A^{0,1}(t)[f,g]}\circ e^{A^{0,0,1}(t)[f,[f,g]]} \circ e^{A^{1,0,1}(t)[g,[f,g]]}(x(0))   
			   +o(t^3)\end{equation}	\end{remark}

		As a straightforward consequence of Proposition \ref{ispiratoaSchattler3},  one gets
		\begin{cor}\label{corollariodischattlerutile3}Let us fix $t>0$. Assume that  $m=1$ and that  the vector fields  $f:=g_0, g=g_1$  belong to  $C^3(\Omega,\rr^n)$. If one chooses the $L^\infty$ control  $a^0, a^1:[0,t]\to U $ in Remark \ref{espansione3}  such that $A^0(t)=A^1(t)=A^{0,1}(t)=A^{0,0,1}(t)=0$ and $A^{1,0,1}(t)=kt^3+o(t^3)$ with $k>0$, then one has $$x(t)-x(0)=ct^3[g,[f,g]](x(0))+o(t^3).$$
		\end{cor}
	}

\section{Proof of Theorem 3.7}\label{proof3.6} 	The proof of {\rm 1)}  in Theorem \ref{eccolestime} is a  standard issue in the proof of the classical Pontryagin maximum Principle, so we skip it.
		\subsubsection{\sc Proof of  {\rm 2)} in Theorem \ref{eccolestime}}    \,
		\noindent
	Let us  recall that $i,j$, with  $0<i<j$}, are fixed.		Observe that \bel{coincidono}x_\varepsilon(\ol{t}) = \hat x(2\sqrt{\varepsilon})
\eeq	where $\hat x$ is  the solution on $[0,2\sqrt{\varepsilon}]$ of the Cauchy problem
	\bel{var1}\left\{\begin{array}{l}\hat{x}(s)= \hat{a}^0(s)\cdot f(\hat{x}(s)) + \sum\limits_{{r=1}}^m \hat{a}^r(s)(s)\cdot{g}_r(\hat{x}(s)),\\
		\hat x(0) = x(\bar t)\end{array}\right.\eeq
		
		{ where the controls $\hat a^h$, $h=0,\ldots,m$, are defined as follows:
{ $$ (\hat a^{0},\hat a^{1},\ldots \hat a^{m})(s):=\begin{cases}
			-\left(1, u(\ol{t}-s)\right)  & \text{ if } s \in [0,\sqrt{\varepsilon}]\\\displaystyle
			\,\,\,\left(1, u(\ol{t}-2\sqrt{\varepsilon}+s)\right)+\gamma_{(i,j)}\left(\frac{ s-\sqrt{\varepsilon}}{{\sqrt{\varepsilon}}}\right) &\text{ if } s \in [\sqrt{\varepsilon},2\sqrt{\varepsilon}]
		\end{cases}$$		
		
		}}

		%
		\noindent
		Indeed, notice that  one has $\hat x(\sqrt{\varepsilon}) = x(\bar t -\sqrt{\varepsilon})$, while  in $[\bar t -\sqrt{\varepsilon}, \bar t]$  we are implementing the corresponding  control ${\bf u}_{{\varepsilon},{\bf c} ,\bar t}$ in \eqref{perGoh}, so obtaining  \eqref{var1}.  
		As in Section \ref{approx.sect}, let us set,
%
		{
			$$
			\ds\hat A^{h}(s) :=\int_0^s \hat a^{h}(\sigma)d\sigma , \quad \hat A^{h,k}(s) :=\int_0^s \hat A^h(\sigma)\hat a^{k}(\sigma)d\sigma\quad \forall h,k =0,\ldots m ,
			$$
	
}
		\noindent
		Let us  check that $\hat{A}^0(2\sqrt{\varepsilon})=0$  and that  the choice  of  the map $\gamma_{(i,j)}$  yields  $\hat{A}^h(2\sqrt{\varepsilon})=0$ $\forall h\in \{1,\dots,m\}$, $\hat{A}^{h,k}(2\sqrt{\varepsilon})=0$ $\forall (h,k)\in \{1,\dots,m\}^2\backslash \{(i,j)\}$,  and $\hat{A}^{i,j}(2\sqrt{\varepsilon})\neq 0$
	(This will allow us to  apply  Corollary \ref{corollariodischattlerutile} and get the desired conclusion).
		%
		For all $s\in[0,2\sqrt{\varepsilon}]$,
		 one has
		$$\hat{A}^0(s)=-s{\bf 1}_ {[0,\sqrt{\varepsilon}]}+(s-2\sqrt{\varepsilon}){\bf 1}_{[\sqrt{\varepsilon},2\sqrt{\varepsilon}]},$$ and, for all $h\in\{1,\ldots,m\}$, 
		{	$$\begin{array}{c}	\hat{A}^{h}(s) = \displaystyle \bigintss_0^s  \left[-\mathbf{1}_{[0,1\sqrt{\varepsilon}]}(\sigma)\cdot u^h(\bar t-\sigma)+ \mathbf{1}_{[\sqrt{\varepsilon},2\sqrt{\varepsilon}]}(\sigma)\cdot \left(u^h(\bar t+\sigma - 2\sqrt{\varepsilon} )+ \gamma_{(j,i)}\left(\frac{ \sigma-\sqrt{\varepsilon}}{{\sqrt{\varepsilon}}}\right)\right)\right]\,d\sigma\end{array} $$}
		Hence,
			\bel{A0}\hat{A}^0(2\sqrt{\varepsilon})=0,
			  \eeq and, by  $\U_{(i,j)}(1) - \U_{(i,j)}(0) =0$, one also has \bel{Ah}\begin{array}{c}	\hat{A}^{h}(2\sqrt{\varepsilon})= \displaystyle \bigintsss_0^{2\sqrt{\varepsilon}} \left[-\mathbf{1}_{[0,1\sqrt{\varepsilon}]}(\sigma)\cdot u^h(\bar t-\sigma)+ \mathbf{1}_{[\sqrt{\varepsilon},2\sqrt{\varepsilon}]}(\sigma)\cdot\left(u^h(\bar t+\sigma - 2\sqrt{\varepsilon} )\right)\right]d\sigma   + \\\\\quad\quad\quad\quad\quad \quad\quad\quad\quad\quad\quad\quad\quad\quad\quad\quad\quad\quad\quad\quad\quad\quad\quad\quad\quad\quad\quad\,\,\,\,\,\,\,\,\,\,\,\,\,\,+\sqrt{\varepsilon} \Big(\U^h_{(i,j)}(1) - \U^h_{(i,j)}(0)\Big) = \\\\
				=	\displaystyle -\bigintsss_0^{\sqrt{\varepsilon}} u^h(\bar t-\sigma)d\sigma +	\displaystyle \bigintsss_0^{\sqrt{\varepsilon}}u^h(\bar t-\sigma)d\sigma 
				  = 0.\end{array}\eeq
	In other words , for any $h,k=1,\ldots,m$ the time-space curves $({\hat A}^0,{\hat A}^h)$ and the space curve $({\hat A}^h,{\hat A}^k)$ are closed, in that $({\hat A}^0,{\hat A}^k)(0) = ({\hat A}^0,{\hat A}^k)(2\sqrt{\varepsilon})=0$ and $({\hat A}^h,{\hat A}^k)(0) = ({\hat A}^h,{\hat A}^k)(2\sqrt{\varepsilon})=0$.		

	In order to compute the coefficients $\hat{A}^{h,k}(2\sqrt{\varepsilon})$ when $h,k=1,\ldots,m$, let us set, for every $h=1,\ldots,m$ and every $s\in [0,2\sqrt{\varepsilon}]	$,
{ $$\check{a}^h_u(s):=
	-u^h(\ol{t}-s)\cdot{\bf 1}_{[0,\sqrt{\varepsilon}]} + 
	u^h(\ol{t}-2\sqrt{\varepsilon}+s)\cdot {\bf 1}_{[\sqrt{\varepsilon},2\sqrt{\varepsilon}]},$$  $$\check{A}^h_u(s) := \ds\int_0^s \check a^h_u(\sigma)d\sigma =\int_0^s\Big(-u^h(\ol{t}-\sigma)\cdot{\bf 1}_{[0,\sqrt{\varepsilon}]} + 
	u^h(\ol{t}-2\sqrt{\varepsilon}+\sigma)\cdot {\bf 1}_{[\sqrt{\varepsilon},2\sqrt{\varepsilon}]} \Big) d\sigma  $$}
and	
{\small$$\begin{array}{c}\displaystyle  \check {a}^{h}(s):=
		\sum\limits_{r=1}^m\gamma^h_{(i,j)}\left(\frac{ s-\sqrt{\varepsilon}}{{\sqrt{\varepsilon}}}\right)\cdot {\bf 1}_{[\sqrt{\varepsilon},2\sqrt{\varepsilon}] },
		\displaystyle \\ \check A^{h}(s) :=\ds\int_0^s\check a^{h}(\sigma)d\sigma  =\left\{\begin{array}{ll} 0  &\forall s\in [0,\sqrt{\varepsilon}]\\  \sqrt{\varepsilon}\left(\Gamma_{(i,j)}^h\left(\frac{ s-\sqrt{\varepsilon}}{{\sqrt{\varepsilon}}}\right)- \Gamma^h_{(i,j)}(0)\right) \,\,\,&\forall s\in [\sqrt{\varepsilon},2\sqrt{\varepsilon}],\end{array}\right.\end{array}.$$}
Hence 
$$
\quad \hat {a}^{h}(s) = \check{a}^h_u(s) + \check{a}^{h}(s)\qquad 	\hat A^{h}(s) = \check{A}^h_u(s) +\check A^{h}(s)\qquad \forall s\in[0,2\sqrt{\varepsilon}]
$$
Observing that   $\ds \int_0^{2\sqrt{\varepsilon}}\check{A}^h_u(s)\check a_{u}^k(s) ds =0$, we get 
\bel{variation21}
\hat{A}^{h,k}(2\sqrt{\varepsilon})=
\ds \int_0^{2\sqrt{\varepsilon}}{\check A^{h}}(s)\check a_{u}^{k}(s)ds+\int_0^{2\sqrt{\varepsilon}}\check{A}^h_u(s)\check a^{k}(s)ds +\int_0^{2\sqrt{\varepsilon}}\check A^{h}(s)\check a^{k}(s)ds \eeq 	
Since the curves $(\hat{A}^0,\hat{A}^k)$, $(\hat{A}^h,\hat{A}^k)$, and $(\check{A}^h_u,\check{A}^k_u)$ are closed, we can give the following interpretation to the some of above coefficients: $$\hat{A}^{0,k}(2\sqrt{\varepsilon}) = Area(\hat{A}^0,\hat{A}^k),\quad \hat{A}^{h,k}(2\sqrt{\varepsilon}) = Area(\hat{A}^h,\hat{A}^k),$$
and $$ 
\check{A}_u^{h,k}(2\sqrt{\varepsilon}):=\int_0^{2\sqrt{\varepsilon}}\check{A}^h_u(s)\check a_{u}^k(s) ds =Area(\check{A}^h_u, \check{A}^k_u) =0	$$		
Let us compute the three terms on the right-hand side of \eqref{variation21}.		 Since $2\sqrt{\varepsilon}$ is a Lebesgue point  for the map 	$$[\sqrt{\varepsilon},2\sqrt{\varepsilon}]\ni s\to \check A^{h}(s)\check a^k_u(s) = 
u^k\left(\bar t  -2\sqrt{\varepsilon}+s \right)\cdot\left(\sqrt{\varepsilon}\int_{0}^{\frac{s-\sqrt{\varepsilon}}{\sqrt{\varepsilon}}}  \gamma^h_{(i,j)}\left(\sigma\right)\,d\sigma \right),  $$
one gets 	$$\begin{array}{c}
	\ds \int_0^{2\sqrt{\varepsilon}}\check A^{h}(s)\check a^k_u(s)ds  =   \int_{\sqrt{\varepsilon}}^{2\sqrt{\varepsilon}}\check A^{h}(s)\check a^k_u(s)ds = 
	\ds\sqrt{\varepsilon}\sqrt{\varepsilon}\left( u^k\left(\bar t\right)\cdot \left(\int_{0}^{1}  \frac{d\Gamma^h_{(j,i)}}{d\sigma}\left(\sigma\right)\,d\sigma\right) \right) ds + o({\varepsilon})= \\   
	\ds {\varepsilon} u^k\left(\bar t  \right) \left(\Gamma^h_{(j,i)}(1)- \Gamma^h_{(j,i)}(0)\right)\,ds + o({\varepsilon}){ = o({\varepsilon})}\qquad \forall h,k\in\{1,\ldots,m\}
\end{array}$$

Moreover, since $\check{A}^k(2\sqrt{\varepsilon})=\check{A}^k(\sqrt{\varepsilon})=0$ for all $k\in\{1,\ldots,m\}$, the above estimate implies 

{	$$\begin{array}{c}
		\ds \int_0^{2\sqrt{\varepsilon}}\check{A}^h_u(s)\cdot \check a^{k}(s)ds  = \ds \int_{\sqrt{\varepsilon}}^{2\sqrt{\varepsilon}}\check{A}^h_u(s)\cdot \check a^{k}(s)ds =\\\ds 
		\left.\check{A}^h_u(s)\cdot \check A^{k}(s)\right|_{\sqrt{\varepsilon}}^{2\sqrt{\varepsilon}} -
		\int_{\sqrt{\varepsilon}}^{2\sqrt{\varepsilon}}\check{a}^h_u(s)\cdot \check{A}^{k}(s)ds= 
		\Big(\check{A}^h_u(2\sqrt{\varepsilon})\cdot \check A^{k}(2\sqrt{\varepsilon})- \check{A}^h_u(\sqrt{\varepsilon})\cdot \check A^{k}(\sqrt{\varepsilon})\Big) + o({\varepsilon}) =  o({\varepsilon})
		
	\end{array}$$}

Finally, for any $h,k=1,\ldots,m$ (see \eqref{areazero1},\eqref{areazero2}),\bel{Ahk}\begin{array}{l} 
	\ds	\hat{A}^{h,k}(2\sqrt{\varepsilon})
	= \int_0^{2\sqrt{\varepsilon}}\check A^{h}(s)\check a^{k}(s)ds =	\int_{\sqrt{\varepsilon}}^{2\sqrt{\varepsilon}}\left(\int_{\sqrt{\varepsilon}}^{s}   \ \frac{d\Gamma^h_{(i,j)}}{dt}\left(\frac{\sigma-\sqrt{\varepsilon}}{\sqrt{\varepsilon}}\right)d\sigma\right)  \frac{d\Gamma^k_{(i,j)}}{dt}\left(\frac{s-\sqrt{\varepsilon}}{\sqrt{\varepsilon}}\right)ds= \\\\
	=\ds {\varepsilon}	\int_{0}^{1} \left(\int_{0}^{s}    \frac{d\Gamma^h_{(i,j)}}{d\sigma}\left(\sigma\right)d\sigma\right) \frac{d\Gamma^k_{(i,j)}}{ds}\left(s\right)ds  = {\varepsilon}	\int_{0}^{1} \left(\Gamma^h_{(i,j)}(s) -\Gamma^h_{(i,j)}(0) \right) \gamma^k_{(i,j)}(s)ds =\\\\
	\ds  {\varepsilon}	\int_{0}^{1} \Gamma^h_{(i,j)}(s) \gamma^k_{(i,j)}(s) ds	 = \left\{\begin{array}{ll}\varepsilon Area\left(\Gamma^h_{(i,j)},\Gamma^k_{(i,j)}\right) = 0 &\text{if} \,\,\,\, (h,k)\neq (i,j) \,\,\,\, \text{or} (h,k)\neq (j,i)\,\,\\
		{\varepsilon} Area\left(\Gamma^i_{(i,j)},\Gamma^j_{(i,j)}\right)={\varepsilon}/{2r^2} \quad &\text{if} \,\,\,\,  (h,k) = (i,j)
		\\	- {\varepsilon} Area\left(\Gamma^i_{(i,j)},\Gamma^j_{(i,j)}\right)=-{\varepsilon}/{2r^2} &\text{if} \,\,\,\,  (h,k) = (j,i)\qquad\footnotemark\end{array}\right.
\end{array}\eeq
\footnotetext{We recall from Definition \ref{signals} that $\ds r:={(\alpha^i)}^{-1}+{(\alpha^j)}^{-1}+{(\beta^i)}^{-1}+{(\beta^j)}^{-1}.$}
%
Moreover, if $k>0$, one has
{\small
\bel{A0k}	\begin{array}{c} \hat{A}^{0,k}(2\sqrt{\varepsilon})  =\qquad\qquad\qquad\qquad\qquad\qquad\qquad\qquad\qquad\qquad\qquad\qquad\qquad\qquad\\\\ \ds\int_0^{\sqrt{\varepsilon}} \sigma \cdot u^i(\bar t-\sigma)\, d\sigma +
				\int_{\sqrt{\varepsilon}}^{2\sqrt{\varepsilon}} (\sigma-2\sqrt{\varepsilon}) \cdot  u^i(\bar t -2\sqrt{\varepsilon}+\sigma )\, d\sigma   +\ds\int_{\sqrt{\varepsilon}}^{2\sqrt{\varepsilon}}\gamma^k_{(i,j)}\left(\frac{ \sigma-\sqrt{\varepsilon}}{{\sqrt{\varepsilon}}}\right) \cdot (\sigma-2\sqrt{\varepsilon})\,d\sigma=\\\\	\ds\int_{\sqrt{\varepsilon}}^{2\sqrt{\varepsilon}}\gamma^k_{(i,j)}\left(\frac{ \sigma-\sqrt{\varepsilon}}{{\sqrt{\varepsilon}}}\right) \cdot (\sigma-2\sqrt{\varepsilon})\,d\sigma
				= \end{array}
		\eeq
		$$	=
	\ds	\varepsilon\int_0^1 \frac{d\Gamma^k_{(i,j)}}{ds}(s)\cdot (s-1)\,ds = \Gamma^k_{(i,j)}(s)\cdot (s-1)\Big|_0^1 	-\varepsilon\int_0^1 \Gamma^k_{(i,j)}(s)ds  = 0 $$}

 By applying Corollary \ref{corollariodischattlerutile}, we then get 
 $$
x_\varepsilon(\ol{t})-x(\ol{t}) = \hat x(2\sqrt{\varepsilon}) -x(\ol{t}) = \varepsilon M \, [g_i,g_j](x(\ol{t})) + o(\varepsilon),
 $$
where $\ds M:=\frac{1}{2r^2}$, so \eqref{2} is proved.
	\subsubsection{\sc Proof of  {\rm 3)} in Theorem \ref{eccolestime}} 
{ Let us fix $i\in{1,\ldots,m}$ and, similarly to the previous step, let us consider again  the solution $\tilde x$  of  Cauchy problem 
		\bel{var2}\left\{\begin{array}{l}\tilde{x}(s)= \tilde{a}^0(s)\cdot f(\tilde{x}(s)) + \sum\limits_{{r=1}}^m \tilde{a}^r(s)(s)\cdot{g}_r(\tilde{x}(s)),\\
		\tilde x(0) = x(\bar t) ,\end{array}\right.\eeq  the controls $\tilde a^h$ being now  defined as 
	{
	$$ (\tilde a^{0},\tilde a^{1},\ldots \tilde a^{m})(s):=\begin{cases}
		-\left(1, u(\ol{t}-s)\right)  & \text{ if } s \in [0,\sqrt{\varepsilon}]\\\displaystyle
		\,\,\,\left(1, u(\ol{t}-2\sqrt{\varepsilon}+s)\right)+ \gamma_{(0,i)}\left(\frac{ s-\sqrt{\varepsilon}}{{\sqrt{\varepsilon}}}\right)&\text{ if } s \in [\sqrt{\varepsilon},2\sqrt{\varepsilon}].
	\end{cases}$$}
Once again,  we will finally  exploit equality \eqref{coincidono}, namely $x_\varepsilon(\ol{t}) = \tilde x(2\sqrt{\varepsilon})$.
Let us set
$$
\ds\tilde A^{h}(s) :=\int_0^s \tilde a^{h}(\sigma)d\sigma , \quad \tilde A^{h,k}(s) :=\int_0^s \tilde A^h(\sigma)\tilde a^{k}(\sigma)d\sigma\quad \forall h,k =0,\ldots m .
$$
As in the previous case we have
$$ \tilde A^{h}(2\sqrt{\varepsilon}) =0  \quad   \forall h=0,\ldots m. $$
Moreover, from  {
$\ds\int_0^1 \U^i_{(0,i)}(s)ds=
 -\frac{\alpha^i \beta^i}{2(\alpha^i+\beta^i)}$    } (see \eqref{oi}),
we get
	$$	\begin{array}{c} \tilde{A}^{0,i}(2\sqrt{\varepsilon})  =\ds\int_0^{\sqrt{\varepsilon}} \sigma \cdot u^i(\bar t-\sigma)\, d\sigma +
				\int_{\sqrt{\varepsilon}}^{2\sqrt{\varepsilon}} (\sigma-2\sqrt{\varepsilon}) \cdot  u^i(\bar t -2\sqrt{\varepsilon}+\sigma )\, d\sigma + \\\quad\quad\quad\quad\quad\quad\quad\quad\quad\quad\quad\quad\quad\quad\quad\quad\quad\quad\quad\quad\quad +\ds\int_{\sqrt{\varepsilon}}^{2\sqrt{\varepsilon}}\gamma^i_{(0,i)}\left(\frac{ \sigma-\sqrt{\varepsilon}}{{\sqrt{\varepsilon}}}\right) \cdot (\sigma-2\sqrt{\varepsilon})\,d\sigma=\\	\ds\int_{\sqrt{\varepsilon}}^{2\sqrt{\varepsilon}}\gamma^i_{(0,i)}\left(\frac{ \sigma-\sqrt{\varepsilon}}{{\sqrt{\varepsilon}}}\right) \cdot (\sigma-2\sqrt{\varepsilon})\,d\sigma
				= 
		\varepsilon\int_0^1 \frac{d\U^i_{(0,i)}}{ds}(s)\cdot (s-1)\,ds = 	-\varepsilon\int_0^1 \U^i_{(0,i)}(s)ds  =  \frac{\alpha^i \beta^i}{2(\alpha^i+\beta^i)}\varepsilon ,    \end{array}
	$$
	 while, for every $ k\neq i$, one has 
$$	\tilde{A}^{0,k}(2\sqrt{\varepsilon})  =\ds\int_0^{\sqrt{\varepsilon}} \sigma \cdot u^i(\bar t-\sigma)\, d\sigma +
	\int_{\sqrt{\varepsilon}}^{2\sqrt{\varepsilon}} (\sigma-2\sqrt{\varepsilon}) \cdot  u^i(\bar t -2\sqrt{\varepsilon}+\sigma )\, d\sigma = 0 .
$$ }
 By applying Corollary \ref{corollariodischattlerutile}, we now get 
$$
x_\varepsilon(\ol{t})-x(\ol{t}) = \tilde x(2\sqrt{\varepsilon}) -x(\ol{t}) = \varepsilon \tilde M \, [g_0,g_i](x(\ol{t})),
$$
where $\ds \tilde M:=	\frac{\beta_1\beta_2}{2(\beta_1+\beta_2)}$, so \eqref{1} is proved.

\subsubsection{\sc Proof of  {\rm 4)} in Theorem \ref{eccolestime}}  Once again we will expolit the equality  \bel{coincidono4}x_\varepsilon(\ol{t}) = \ol{x}(2\sqrt{\varepsilon})
\eeq	where $\ol{x}$ is  the solution on $[0,2\sqrt{\varepsilon}]$ of the Cauchy problem
		\bel{var2}\left\{\begin{array}{l}\ol{x}(s)= \ol{a}^0(s)\cdot f(\ol{x}(s)) +  \ol{a}^1(s)\cdot g_1(\ol{x}(s)),\\
		\ol{x}(0) = x(\bar t) ,\end{array}\right.\eeq  the controls $(\ol{a}^0,\ol{a}^1)$ being now  defined as 
	$$ (\ol{a}^{0},\ol{a}^{1})(s):=\begin{cases}
		-\left(1, u(\ol{t}-s)\right)  & \text{ if } s \in [0,\sqrt{\varepsilon}]\\\displaystyle
		\,\,\,\left(1, u(\ol{t}-2\sqrt{\varepsilon}+s)\right)+ \gamma_{(101)}\left(\frac{ s-\sqrt{\varepsilon}}{{\sqrt{\varepsilon}}}\right) &\text{ if } s \in [\sqrt{\varepsilon},2\sqrt{\varepsilon}]
	\end{cases}$$

Let us set
$$
\ds\ol{A}^{h}(s) :=\int_0^s\ol{a}^{h}(\sigma)d\sigma , \quad \quad \forall h =0,1\qquad \ol{A}^{0,1}(s) :=\int_0^s \ol{A}^0(\sigma) \ol{a}^{1}(\sigma)d\sigma,
$$
$$\ol{A}^{101}(s) := \ds\int_0^s \ol{A}^1(\sigma)\frac{d\ol{A}^{0,1}}{d\sigma}(\sigma)d\sigma  = \int_0^s \ol{A}^1(\sigma) \ol{A}^0(\sigma)a^1(\sigma)d\sigma,$$
$$\ol{A}^{001}(s) := \ds\int_0^s \ol{A}^0(\sigma)\frac{d\ol{A}^{0,1}}{d\sigma}(\sigma)d\sigma  = \int_0^s \ol{A}^0(\sigma) \ol{A}^0(\sigma)a^1(\sigma)d\sigma.$$
As in the previous case one has
$$
\ol{A}^{0}(2\sqrt{\varepsilon}) = 0,
$$
$$
\begin{array}{l}\ds\ol{A}^{1}(2\sqrt[3]{\varepsilon}):=\int_0^{2\sqrt[3]{\varepsilon}} \ol{a}^{1}(s)ds=-\int_0^{\sqrt[3]{\varepsilon}} u(\ol{t}-s)ds+\int_{\sqrt[3]{\varepsilon}}^{2\sqrt[3]{\varepsilon}}u(\ol{t}-2\sqrt[3]{\varepsilon}+s) ds+\\\ds\qquad\qquad+\int_{\sqrt[3]{\varepsilon}}^{2\sqrt[3]{\varepsilon}} \gamma_{1,0,1}^1 \left( \frac{s-\sqrt[3]{\varepsilon}}{\sqrt[3]{\varepsilon}}  \right) ds={ \sqrt[3]{\varepsilon}} \int_{0}^{1} \gamma_{1,0,1}^1 ( {\sigma}) d\sigma  = { \sqrt[3]{\varepsilon}} \Big(\Gamma_{1,0,1}^1(1) -\Gamma_{1,0,1}^1(0)\Big)=0.  \end{array}$$
Moreover, $\ol{A}^{0,1}(2\sqrt[3]{\varepsilon}) = \ol{A}^{0,0,1}(2\sqrt[3]{\varepsilon}) =0$. Indeed,
$$\begin{array}{l} \ds\ol{A}^{0,1}(2\sqrt[3]{\varepsilon}) :=\int^{2\sqrt[3]{\varepsilon}}_{0} \ol{A}^0(s) \ol{a}^{1}(s) ds=\\

\\ =\ds \int_{0}^{\sqrt[3]{\varepsilon}} s u(\ol{t}-s) ds  +  \ds\int_{\sqrt[3]{\varepsilon}}^{2\sqrt[3]{\varepsilon}} (s-\sqrt[3]{\varepsilon}) u(\ol{t}-2\sqrt[3]{\varepsilon}+s)ds +  \int_{\sqrt[3]{\varepsilon}}^{2\sqrt[3]{\varepsilon}} (s-\sqrt[3]{\varepsilon})\gamma_{1,0,1}^1 \left( \frac{s-\sqrt[3]{\varepsilon}}{\sqrt[3]{\varepsilon}}  \right) ds=\\ =\ds  \sqrt[3]{\varepsilon^2} \int_{0}^{1} \sigma\gamma_{1,0,1}^1(\sigma) ds =\sqrt[3]{\varepsilon^2}\left( 1\cdot\Gamma_{1,0,1}^1(1) -0\cdot \Gamma_{1,0,1}^1(0) - \int_0^1\Gamma_{1,0,1}^1(s)ds \right)= 0 \end{array}
$$
and
$$
\begin{array}{l} \ds\ol{A}^{0,0,1}(2\sqrt[3]{\varepsilon}) :=\int^{2\sqrt[3]{\varepsilon}}_{0} \left(\ol{A}^0(s)\right)^2 \ol{a}^{1}(s) ds =\\ \ds -\int_{0}^{\sqrt[3]{\varepsilon}} s^2 u(\ol{t}-s) ds  +  \ds\int_{\sqrt[3]{\varepsilon}}^{2\sqrt[3]{\varepsilon}} (s-\sqrt[3]{\varepsilon})^2 u(\ol{t}-2\sqrt[3]{\varepsilon}+s)ds +  \int_{\sqrt[3]{\varepsilon}}^{2\sqrt[3]{\varepsilon}} (s-\sqrt[3]{\varepsilon})^2\gamma_{1,0,1}^1 \left( \frac{s-\sqrt[3]{\varepsilon}}{\sqrt[3]{\varepsilon}}  \right) ds=\\ =\ds  \sqrt[3]{\varepsilon^2} \int_{0}^{1} \sigma^2\gamma_{1,0,1}^1(\sigma) ds =\sqrt[3]{\varepsilon^2}\left( 1\cdot\Gamma_{1,0,1}^1(1) -0\cdot \Gamma_{1,0,1}^1(0) - \int_0^1\Gamma_{1,0,1}^1(s)ds \right)= 0.\end{array}
$$
Finally,
\small
$$
\begin{array}{l}\qquad\qquad\qquad\qquad \ds\ol{A}^{1,0,1}(2\sqrt[3]{\varepsilon}) =
\ds \int_{0}^{2 \sqrt[3]{\varepsilon}} (\ol{A}^1)^{2} \ol{a}^0 d\sigma = 
  \\
= - \ds \bigintsss_{0}^{\sqrt[3]{\varepsilon}} \Big(\int_{0}^{\eta} u(t-\sigma)  d\sigma \Big)^2d\eta + \ds \ \bigintsss_{\sqrt[3]{\varepsilon}}^{2\sqrt[3]{\varepsilon}} \Bigg( -  \int_{\eta}^{2\sqrt[3]{\varepsilon}} u(t-\sigma-2\sqrt[3]{\varepsilon}) d\sigma + \int_{\sqrt[3]{\varepsilon}}^{\eta} \gamma_{1,0,1}^1\Big(\frac{\sigma-\sqrt[3]{\varepsilon}}{\sqrt[3]{\varepsilon}}\Big) d\sigma \Bigg)^2 d\eta =  \\
  =- \ds \ \bigintsss_{0}^{\sqrt[3]{\varepsilon}} \Big(\int_{0}^{\eta} u(t-\sigma)  d\sigma \Big)^2d\eta +
 \int_{\sqrt[3]{\varepsilon}}^{2\sqrt[3]{\varepsilon}} \Bigg(\int_{\eta}^{2\sqrt[3]{\varepsilon}} u(t-\sigma-2\sqrt[3]{\varepsilon})\Bigg)^2 d\sigma +
  \\
  \ds +\ \bigintsss_{\sqrt[3]{\varepsilon}}^{2\sqrt[3]{\varepsilon}} \Bigg(   - 2\int_{\eta}^{2\sqrt[3]{\varepsilon}} u(t-\sigma-2\sqrt[3]{\varepsilon}) d\sigma \int_{\sqrt[3]{\varepsilon}}^{\eta} \gamma_{1,0,1}^1\left(\frac{\sigma-\sqrt[3]{\varepsilon}}{\sqrt[3]{\varepsilon}}\right) d\sigma + \Big( \int_{\sqrt[3]{\varepsilon}}^{\eta} \gamma_{1,0,1}^1\left(\frac{\sigma-\sqrt[3]{\varepsilon}}{\sqrt[3]{\varepsilon}}\right) d\sigma \Big) ^2 \Bigg) d\eta = \\=  \ds \ \bigintsss_{\sqrt[3]{\varepsilon}}^{2\sqrt[3]{\varepsilon}} \Bigg(   - 2\int_{\eta}^{2\sqrt[3]{\varepsilon}} u(t-\sigma-2\sqrt[3]{\varepsilon}) d\sigma \int_{\sqrt[3]{\varepsilon}}^{\eta} \gamma_{1,0,1}^1\Big(\frac{\sigma-\sqrt[3]{\varepsilon}}{\sqrt[3]{\varepsilon}}\Big) d\sigma + \Big( \int_{\sqrt[3]{\varepsilon}}^{\eta} \gamma_{1,0,1}^1\Big(\frac{\sigma-\sqrt[3]{\varepsilon}}{\sqrt[3]{\varepsilon}}\Big) d\sigma \Big) ^2 \Bigg) d\eta = \end{array}$$\footnote{The symbol $\ds\fint$ denotes the integral average, namely $\ds \fint_a^b\varphi(t)dt:= \frac{\int_a^b\varphi(t)dt}{b-a}.$}
$$\begin{array}{l}
 = \ds 2  \bigintsss_{\sqrt[3]{\varepsilon}}^{2\sqrt[3]{\varepsilon}}  \left(   - \fint_{\eta}^{2\sqrt[3]{\varepsilon}} u(t-\sigma-2\sqrt[3]{\varepsilon}) d\sigma \right) \cdot \Bigg(2 \sqrt[3]{\varepsilon} - \eta    \Bigg) \cdot \Bigg(\int_{\sqrt[3]{\varepsilon}}^{\eta} \gamma_{1,0,1}^1\left(\frac{\sigma-\sqrt[3]{\varepsilon}}{\sqrt[3]{\varepsilon}}\right) d\sigma \Bigg) d\eta +\\ \qquad\qquad\qquad+\ds\bigintsss_{\sqrt[3]{\varepsilon}}^{2\sqrt[3]{\varepsilon}} \Bigg( {\sqrt[3]{{\varepsilon}}} \Bigg( \Gamma_{1,0,1}^1 \Big(\frac{\eta-\sqrt[3]{\varepsilon}}{\sqrt[3]{\varepsilon}}\Bigg) - \Gamma_{1,0,1}^1(0) \Big)\Bigg)^2 d\eta = \\ 
=\ds   2  \bigintsss_{\sqrt[3]{\varepsilon}}^{2\sqrt[3]{\varepsilon}} \left[\left(-\fint_{\eta}^{2\sqrt[3]{\varepsilon}}  u(t-\sigma-2\sqrt[3]{\varepsilon}) d\sigma\right)\cdot  \Bigg( 2 \sqrt[3]{\varepsilon^2}\int_0^{\frac{\eta-\sqrt[3]{\varepsilon}}{\sqrt[3]{\varepsilon}}}\ \gamma_{1,0,1}^1(s)  ds  - \sqrt[3]{\varepsilon}\eta \int_0^{\frac{\eta-\sqrt[3]{\varepsilon}}{\sqrt[3]{\varepsilon}}}\ \gamma_{1,0,1}^1(s)  ds  \Bigg)\right]d\eta  + \\\qquad\qquad\qquad\ds
  +\varepsilon \bigintsss_{0}^{1} \left(\Gamma_{1,0,1}^1(s)\right)^2 ds  =\\= \ds    -4 \sqrt[3]{\varepsilon^2} \bigintsss_{\sqrt[3]{\varepsilon}}^{2\sqrt[3]{\varepsilon}} \left[\left(\fint_{\eta}^{2\sqrt[3]{\varepsilon}}  u(t-\sigma-2\sqrt[3]{\varepsilon}) d\sigma\right) \cdot   \Gamma_{1,0,1}^1\left( \frac{\eta-\sqrt[3]{\varepsilon}}{\sqrt[3]{\varepsilon}}\right)\right] d\eta-   \\ \ds- 2 \sqrt[3]{\varepsilon}  \bigintsss_{\sqrt[3]{\varepsilon}}^{2\sqrt[3]{\varepsilon}} \left[\eta \left( - \fint_{\eta}^{2\sqrt[3]{\varepsilon}}  u(t-\sigma-2\sqrt[3]{\varepsilon}) d\sigma \right) \cdot  \Gamma_{1,0,1}^1\left( \frac{\eta-\sqrt[3]{\varepsilon}}{\sqrt[3]{\varepsilon}}\right)\right] d\eta  + \varepsilon \bigintsss_{0}^{1} \left(\Gamma_{1,0,1}^1(s)\right)^2 ds  = \\ =\ds - 4 \varepsilon   \bigintsss_{0}^{1} \left[ \left(\fint_{(s+1)\sqrt[3]{\varepsilon}}^{2\sqrt[3]{\varepsilon}}  u(t-\sigma-2\sqrt[3]{\varepsilon}) d\sigma  \right) \cdot \Gamma_{1,0,1}^1\Big( s \Big)\right] ds +  \\\ds + 2 \varepsilon   \bigintsss_{0}^{1} (s+1)\left[ \left(\fint_{(s+1)\sqrt[3]{\varepsilon}}^{2\sqrt[3]{\varepsilon}}  u(t-\sigma-2\sqrt[3]{\varepsilon}) d\sigma  \right) \cdot \Gamma_{1,0,1}^1\Big( s \Big)\right] ds + \varepsilon \bigintsss_{0}^{1} \left(\Gamma_{1,0,1}^1(s)\right)^2 ds= \end{array}$$ 
$$\begin{array}{l} = \ds- 2    \varepsilon   \bigintsss_{0}^{1} \left[ \left(\fint_{(s+1)\sqrt[3]{\varepsilon}}^{2\sqrt[3]{\varepsilon}}  u(t-\sigma-2\sqrt[3]{\varepsilon}) d\sigma  \right) \cdot \Gamma_{1,0,1}^1\Big( s \Big)\right] ds +\\ \ds+ 2 \varepsilon  \bigintsss_{0}^{1} \left[\left( \fint_{(s+1)\sqrt[3]{\varepsilon}}^{2\sqrt[3]{\varepsilon}}  u(t-\sigma-2\sqrt[3]{\varepsilon}) d\sigma\right) \cdot s \cdot\Gamma_{1,0,1}^1\Big(s \Big)  \right]ds 
  + \varepsilon \bigintsss_{0}^{1} \left(\Gamma_{1,0,1}^1(s)\right)^2 ds =\end{array}$$ 
\centerline{\footnotesize(using  integration by parts and the equality $\Gamma_{1,0,1}^1(1)=\Gamma_{1,0,1}^1(0) =0$)}
$$\begin{array}{l}
 \ds =- 4    \varepsilon   \bigintsss_{0}^{1} \left[ \left(\fint_{(s+1)\sqrt[3]{\varepsilon}}^{2\sqrt[3]{\varepsilon}}  u(t-\sigma-2\sqrt[3]{\varepsilon}) d\sigma  \right) \cdot \Gamma_{1,0,1}^1\Big( s \Big)\right] ds 
   + \varepsilon \bigintsss_{0}^{1} \left(\Gamma_{1,0,1}^1(s)\right)^2 ds 
\end{array} $$

Notice that  $$\ds \lim_{\varepsilon\to 0} \frac{1}{\sqrt[3]{\varepsilon}-s\sqrt[3]{\varepsilon}} \int_{(s+1)\sqrt[3]{\varepsilon}}^{2\sqrt[3]{\varepsilon}} \Big( u(t-\sigma-2\sqrt[3]{\varepsilon}) -u(t)\Big) d\sigma  = 0 \quad\implies\quad \fint_{(s+1)\sqrt[3]{\varepsilon}}^{2\sqrt[3]{\varepsilon}}  u(t-\sigma-2\sqrt[3]{\varepsilon}) d\sigma  \Big) = u(t) + o(1).$$ \\

Therefore,
\begin{equation}
\ol{A}^{1,0,1}(2\sqrt[3]{\varepsilon}) = - 4 \varepsilon \Big(u(t) + o(1)\Big) \int_{0}^{1} \Gamma_{1,0,1}^1\Big( s \Big) ds  +  \varepsilon \int_{0}^{1} \left(\Gamma_{1,0,1}^1(s)\right)^2 ds = \varepsilon   \int_{0}^{1} \left(\Gamma_{1,0,1}^1(s)\right)^2 ds
	\end{equation}

In view of \eqref{areazero5} and  Corollary \ref{corollariodischattlerutile3}, we then deduce that 
{ $$
	x_\varepsilon(t) -x(t) = k\varepsilon [g,[f,g]](x(t)) + o(\varepsilon)
	$$}
with 
{
$$k:=\int_{0}^{1} \left(\Gamma_{1,0,1}^1(s)\right)^2 ds  =  \ds\frac23\Big((\beta^1)^2    \Big(\frac{\alpha^1}{2(\alpha^1+\beta^1)}\Big)^3  +(\alpha^1)^2  \Big(\frac{\alpha^1+2\beta^1}{2(\alpha^1+\beta^1)} \Big)^3 
$$}

Hence,  equality \eqref{3} is proved, with 
{
 $$\ds\check M = k = \ds\frac23\Big((\beta^1)^2    \Big(\frac{\alpha^1}{2(\alpha^1+\beta^1)}\Big)^3  +(\alpha^1)^2  \Big(\frac{\alpha^1+2\beta^1}{2(\alpha^1+\beta^1)} \Big)^3  .$$   
}

\section{Proof of Theorem 2.3}\label{setsepsec}

\subsection{Local set separation}

\begin{definition}[Local separation of sets]
	Two subsets $E_1$ and $E_2$ are {\rm locally separated at $x$} if  there exists a neighbourhood $W$ of $x$ such that $\displaystyle E_1\cap E_2 \cap W =\{x\}.$ 
\end{definition}
Let us recall the notion of (Boltyanski) {\it approximating cone} and its relation with local set separation.

\begin{definition} Let  $V$ be  a finite-dimensional real vector space. A subset $C\subseteq V$ is called a {\rm cone} if $\alpha v\in C, \forall \alpha\ge 0$ and $\forall v \in C.$ 
	For any given subset $E \subseteq V$, the set $E^\perp:=\{v \in \rr^n, \, v \cdot c \le 0 \,\,\forall c \in C\}\subseteq V^*$ \footnote{If $V$ is a finite-dimensional real vector space we use $V^*$ to denote its  dual space.} is a closed cone, called the {\rm polar cone of $E$}.
	We say that two cones $C_1$, $C_2$ are {\rm linearly { separable} if  $C_1^\perp \cap -C_2^{\perp}\supsetneq \{0\}$}, that is,  if there exists a linear form $\mu\in V^*\backslash\{0\}$ such that $\mu c_1\geq 0, \mu c_2\leq 0$ for all $(c_1,c_2)\in C_1\times C_2$. \end{definition}

\begin{definition}
	Two { convex} cones $C_1, C_2$ of a vector space $V$  are said to be {\rm transversal} if $$C_1-C_2:=\big\{c_1-c_2, \quad (c_1,c_2)\in C_1\times C_2\big\} = V.$$
	Two transversal cones $C_1$, $C_2$ are called {\rm strongly transversal} if $C_1\cap C_2 \supsetneq \{0\} $. \end{definition}

One easily checks the following equivalences for a pair of cones $C_1$ and $C_2$:\begin{itemize}
	\item{\it $C_1$ and $C_2$ are linearly separable   if and only if   they are not transversal}.
	\item{\it  $C_1$ and $C_2$  are strongly transversal if and only if they are transversal and  there exist a non-zero linear form $\mu$  and an element $c\in C_1\cap C_2$  such that $\mu (c)>0$. }
\end{itemize} 

\begin{definition} [Boltyanski approximating cone]
	Let $Z$ be a subset of $\mathbb{R}^n$ for some integer $n \ge 1$. Fix $z \in Z$.  We say that
	a convex cone $K \subseteq \rr^N$ is a Boltyanski approximating cone for $Z$ at $z$ if there exist a convex
	cone $C \subseteq \mathbb{R}^m$ for some integer $m>0$, a neighborhood $U$ of $0$ in $\mathbb{R}^m$, and a continuous
	map $F : U \cap C \to Z$ such that:
	
	i) $F(0) = z$;
	
	ii)   There exists a linear map $L : \mathbb{R}^m \to \mathbb{R}^n$ verifying 
	$F(v) = F(0) + Lv + o(|v|) \,\,\,\,\,\,\forall v \in U \cap C$; 
	
	iii) $LC = K.$	 
\end{definition}

The following open-mapping-based result characterizes set-separation in terms of linear separation of approximating cones (see e.g. \cite{setsep}).
\begin{theorem}[Set separation of approximating cones]\label{th-set-sep}
	Let $Z_1$ and $Z_2$ be subsets of $\mathbb{R}^n, z \in Z_1 \cap Z_2$ and let $K_1, K_2 \subseteq \mathbb{R}^n$ be 
	Boltyanski approximating cones for $Z_1$ and $Z_2$, respectively, at $z$. If $K_1$ or $K_2$ is not a
	subspace and $Z_1, Z_2$ are locally separated at $z$, then $K_1$ and $K_2$ are linearly separated,
	namely there exists a covector $\lambda \in R^n$ such that $0 \neq \lambda \in K_1^{\perp} \cap K_2^{\perp}$. 
	%
\end{theorem}

\subsection{Finitely many variations}
Let us use  $(0,T)_{\mathrm{Leb}}\subset [0,T]$ to denote  the full-measure subset of Lebesgue points  of the $L^1$ map $\ds[0,T]\ni t\mapsto F(t):=\ f(\ol{x}(t))+\sum\limits_{i=1}^m g_i(\ol{x}(t))\ol{u}^i(t)$.
By  Lusin's Theorem  there exists a sequence of subsets $E_q\subset [0,\bar T]$, $q\geq 0$, such that
i) $E_0$ has null measure,
ii) for every $q>0$  $E_q$ is  a compact set such that the restriction of $F$ to $E_q$ is continuous, and iii)
$(0,T)_{\mathrm{Leb}}=\displaystyle\bigcup\limits_{q=0}^{+\infty} E_q.$
For every $q>0$ let us use $D_q\subseteq E_q$  to denote  the set of all density points of $E_q$\footnote{A point $x$ is called a {\it density poin}t for a Lebesgue\textendash{}measurable set $E$ if $\ds\lim\limits_{\rho \to 0} \frac{|B_\rho(x)\cap E|}{|B_\rho(x)|}=1.$}, which,  by  Lebesgue's Theorem  has the same Lebesgue measure as $E_q$. In particular, the subset  $D:=\bigcup\limits_{q=1}^{+\infty} D_q\subset [0,T]$ is full-measure, i.e. it  has measure equal to $T$.\\

For any $(\ol{t},\mathbf{c})\in ]0,{T}] \times \mathfrak{V}$  and any $\varepsilon$ sufficiently small, consider the operator  $$\mathcal{A}_{\varepsilon,\mathbf{\bm{c}},\tau}:
L^\infty([0,T],\rr^m)\to  L^\infty([0,T],\rr^m)  \qquad \mathcal{A}_{\varepsilon,\mathbf{c},\tau}({u}):={u}_{\varepsilon,\mathbf{c},\tau}.$$ 

Clearly, {\it the control  $\mathcal{A}_{\varepsilon,\mathbf{c},\tau}(\ol{u}) ={u}_{\varepsilon,\mathbf{c},\tau}$ might be  not admissible }, namely it can happen that  $\mathcal{A}_{\varepsilon,\mathbf{c},\tau}(\ol{u})({\mathcal T})\not\subset U$ for some subset ${\mathcal T}\subset [0,T]$  having positive measure. 
To avoid this drawback, let us define the following subsets of ${\mathfrak{V}}$:
\bel{signals1}\begin{array}{l}
 {\mathfrak{V}}_{Goh}^{\ol{u}} := \Big\{\mathbf{c}=(i,j)\in {\mathfrak{V}}_{Goh},  \,\,\, \ol{u}\,\,\,\text{ is $i$-balanced and $j$-balanced a.e. }     \Big\}\\
  {\mathfrak{V}}_{LC2}^{\ol{u}} := \Big\{\mathbf{c}=(0,1)\in {\mathfrak{V}}_{LC2},  \,\,\, \ol{u}\,\,\,\text{ is $i$-balanced a.e.} \Big\} \\
 {\mathfrak{V}}_{LC3}^{\ol{u}} := \Big\{\mathbf{c}=(1,0,1),   \,\,\,m=1  \,\,\, \ol{u}\,\,\,\text{ is $1$-balanced a.e.}\Big\} \\ 	
\end{array}\eeq  and
\bel{signals2}
{\mathfrak{V}}^{\ol{u}} := {\mathfrak{V}}_{ndl} \cup {\mathfrak{V}}_{Goh}^{\ol{u}}\cup  {\mathfrak{V}}_{LC2}^{\ol{u}} \cup  {\mathfrak{V}}_{LC3}^{\ol{u}} \qquad\footnotemark
\eeq
\footnotetext{Depending on the hypotheses {\bf v)}-{\bf vii)}  some or even all subsets  ${\mathfrak{V}}_{Goh}^{\ol{u}}$, ${\mathfrak{V}}_{LC2}^{\ol{u}}$,  ${\mathfrak{V}}_{LC3}^{\ol{u}}$ might be empty. However ${\mathfrak{V}}^{\ol{u}}$ is never empty, in that ${\mathfrak{V}}^{\ol{u}}\supseteq  {\mathfrak{V}}_{ndl} (=U)$.  }

Let us consider the full-measure subset $\ds
\Lambda:=\underset{r}{\bigcap} \Lambda_r \subset ]0,T[,
$ where $\Lambda_r\subseteq ]0,T[$ is the (full-measure) subset in Definition \ref{fit2}, and the intersection
 is extended to all $r$ such that $r=i$ or $r\in\{i,j\}$, for some of the indexes $i,j$ appearing in hypotheses ${\bf v)}-\bf{vii)}$ of Theorem \ref{TeoremaPrincipale}. \footnote{In particular, $\underset{r}{\bigcap} \Lambda_r$ is a {\it finite }intersection, so that $\Lambda$ is full-measure, namely $meas(\Lambda) = meas([0,T]) = T$.}
Let $N$ be a natural number  and let us consider  $N$ variation signals $\mathbf{c}_1,\ldots,\mathbf{c}_N \in \mathfrak{V}^{\ol{u}}$ and $N$ instants $ 0<{t}_{1} <  \ldots {t}_N < {T}$, with $t_{h}\in D\cap \Lambda$. For some $\tilde\varepsilon>0$, let us define the {\it  multiple variation} of $\ol{u}$  $$[0,\tilde\varepsilon]^N\ni \small{\bm{\varepsilon}}\mapsto {\ol{u}}_{\bm{\varepsilon}}:=\mathcal{A}_{\varepsilon_N,\mathbf{c}_N,{t}_N} \circ \ldots \circ \mathcal{A}_{{\varepsilon_1},\mathbf{c}_1,{t}_1} ({\ol{u}}).$$ Notice that, in view of the hypotheses of the theorem,  ${\ol{u}}_{\bm{\varepsilon}}$ turns out to be an admissible control  (i.e. ${\ol{u}}_{\bm{\varepsilon}}(t)\in U$ for almost every $t\in [0,T]$) as soon as  $\tilde\varepsilon$ is sufficiently small.
Let $({{\ol{u}}}_{\bm{\varepsilon}},{x}_{\bm{\varepsilon}})$ be the  process corresponding to the control ${{\ol{u}}}_{\bm{\varepsilon}}$.\footnote{Of course, $({\ol{u}}_{\bm{\varepsilon}},{x}_{\bm{\varepsilon}})$  depends on the parameters  $\mathbf{c}_k$ and ${t}_k$ as well.}

The effect of multiple perturbations consists in the sum of  single perturbations, as stated in the following  elementary result: 
\begin{lem}\label{lemmapseudoaffine} 
	The map
	$\ds\bm{\varepsilon}\mapsto
	x_{\bm{\varepsilon}}({T})
	$ from $[0,\tilde \varepsilon]^N$ into  $\mathbb{R}^{n}$  satisfies  
	\begin{equation}\label{pseudoaffine2}
		x_{\bm{\varepsilon}}({T})-x_{0}({T})=\sum\limits_{i=1}^{N} \Big(	x_{\varepsilon_i \mathbf{e}_i}({T}) - x_{0}({T})\Big)+ o(|\bm{\varepsilon}|),\qquad \forall \bm{\varepsilon} =(\varepsilon^1,\dots,\varepsilon^N)\in [0,\tilde \varepsilon]^N.\end{equation}
\end{lem}

In relation with  Theorem \ref{eccolestime}, let us adopt the notation 
\bel{notation}
v_{\mathbf{c},t}:= \left\{\begin{array}{ll}\ds\sum_{r=1}^ mg_\ell(x({t}))\big(\bar u^r - u^r({t})\big)\quad &\text{if}\,\,\mathbf{c}=\bar u \in \mathfrak{V}_{ndl}\\\\
	\,	[g_i,g_j](x({t})) \quad &\text{if} \,\,\mathbf{c}=(i,j) \in \mathfrak{V}_{Goh}\\\\ \,	[f,g_i](x({t})) \quad &\text{if} \,\, \mathbf{c}=(0,i) \in \mathfrak{V}_{LC2}\\\\ \,	[g,[f,g]](x(\ol{t}))   \quad   &\text{if} \,\, m=1\,\,\,\text{and}\,\,\,  \{\mathbf{c}\} = \{(1,0,1)\} =\mathfrak{V}_{LC3}\end{array}\right.
\eeq  Using the foundamental matrix $M(\cdot,\cdot)$  of the the variational equation\footnote{Namely, $t\mapsto M(t,T)$ is the (matrix) solution of the variational Cauchy problem
	$$\frac{dM}{dt}(t,T) = \frac{\partial}{\partial x}\left(f(x)+\sum\limits_{r=1}^m g_i(x)\ol{u}^r(t)\right)_{x=\ol{x}(t)}\cdot \frac{dM}{dt}(t,T) \qquad M(T,T) = Id_{\rr^n}.   $$}
\bel{Var}
\dot v(t) = \frac{\partial}{\partial x}\left(f(x)+\sum\limits_{r=1}^m g_i(x)\ol{u}^r(t)\right)_{x=\ol{x}(t)}\cdot v(t)
\eeq 
associated to the state equation, for every $r=1,\ldots,N$ we deduce the  first order approximation
$$
x_{\varepsilon_r \mathbf{e}_i}({T}) - x_{0}({T}) =\varepsilon_rM(T,t_r)\cdot v_{\mathbf{c_r},t_r} + o(|\bm{\varepsilon}|)
$$  
Hence, by Lemma \ref{lemmapseudoaffine}  we get the following fact:
\begin{cor}\label{varcor}
		The map
	$$\bm{\varepsilon}\mapsto
	x_{\bm{\varepsilon}}({T})
	$$ from $[0,\tilde \varepsilon]^N$ into  $\mathbb{R}^{n}$  satisfies  
	\begin{equation}\label{pseudoaffine2}
		x_{\bm{\varepsilon}}({T})-x_{0}({T})=\sum\limits_{r=1}^{N}\varepsilon_rv_{\mathbf{c_r},t_r} + o(|\bm{\varepsilon}|),\qquad \forall \bm{\varepsilon}=(\varepsilon^1,\dots,\varepsilon^N)\in [0,\tilde \varepsilon]^N.\end{equation}
	\end{cor}
This allows us to build a Boltyanski approximation cone  at $\ds {\begin{pmatrix} x(T)\\\Psi(x(T))\end{pmatrix}}$ to the {\it $\delta\textendash{}$ reachable set} (for some $\delta>0$)
 $$\mathcal{R}_\delta=:\left\{{\begin{pmatrix} x(T)\\\Psi(x(T))\end{pmatrix}},\,\,\exists \,\,\text{an admissible control} \,\,\,u \,\,\text{s.t.}\,\, \left(u,x\right) \text{is a process  s.t.}   \left\|x
-\ol{ x}
\right\|_{C^0}+\|u-\ol{u}\|_{1}<\delta\right\} \subset \rr^{n+1}.$$

Indeed, Corollary \ref{pseudoaffine2} can be rephrased by stating that 
\begin{lem}\label{apconelem} Let us choose $\delta>0$.
 The cone $$B := L \cdot [0,+\infty)^N \subset\rr^{n+1},$$  
		where the homomorphism $L\in Lin(\rr^N,\rr^{n+1})$ is defined by setting  $$L(\varepsilon_1,\ldots,\varepsilon_N)   : =
		\sum_{r=1}^{N}\begin{pmatrix} M(t_r,T) \cdot v_{\mathbb{c}_r,t_r} \varepsilon_r\\\\\ds\frac{\partial \Psi}{\partial x}(\ol{x}(T))\cdot M(t_r,T) \cdot v_{\mathbb{c}_r,t_r}\varepsilon_r\end{pmatrix}
	\qquad \forall (\varepsilon_1,\ldots,\varepsilon_N)\in \rr^N ,	$$ is a Boltyanski approximating cone at ${\begin{pmatrix} \ol{x}(T)\\\Psi(x(T))\end{pmatrix}}$  of the {\rm $\delta\textendash{}$ reachable set}   $\mathcal{R}_\delta$ 
in the direction of $[0,+\infty[^N$.  
\end{lem}

Now, we exploit the fact that, by the  definition of  local weak minimizer, \begin{itemize}
\item{ \it the sets $\mathcal{R}_\delta$  and the {\rm profitable set} $\mathcal{P}:=\Big(\mathfrak{T}\,\,\times\,\, \big]-\infty,\Psi\big(\ol{x}(T)\big)\big[\Big) \bigcup \left\{\begin{pmatrix}\ol{x}(T)\\\Psi\big(\ol{x}(T)\big)\end{pmatrix}
\right\}\subset \Bbb R^{n+1}  $ are locally separated at { { $\begin{pmatrix}\ol{x}(T)\\\Psi\big(\ol{x}(T)\big)\end{pmatrix}$}}.\footnotemark}
\end{itemize}

Therefore, since ${C}\times  ]-\infty,0]$ is a Boltyanski approximating cone of $\mathcal{P}$ at ${\begin{pmatrix} \ol{x}(T)\\\Psi(x(T))\end{pmatrix}}$,  in view of Theorem \ref{th-set-sep}, the  cones $B$ and ${C}\times  ]-\infty,0]$ are not strongly transversal, i.e. they are linearly separable. 	
%
%
\footnotetext{ We use the name {\it profitable set} because this set is made of points which at the same time are admissible and have a cost which is less than or equal to the optimal cost.}
In other words, there exists an adjoint vector $(\xi,\xi_c) \in -\left(C \times ]-\infty,0]\right)^\perp= -C^\perp\times ]-\infty,0]$   verifying 
$$\begin{array}{c} (\xi,\xi_c)\cdot L(\varepsilon_1,\ldots,\varepsilon_N) = \\\\
\ds=	\sum_{r=1}^{N} (\xi,\xi_c)\cdot \begin{pmatrix} M(t_r,T) \cdot v_{\mathbb{c}_r,t_r} \varepsilon_r\\\\\ds\frac{\partial \Psi}{\partial x}(\ol{x}(T))\cdot M(t_r,T) \cdot v_{\mathbb{c}_r,t_r}\varepsilon_r\end{pmatrix}  = 	\sum_{r=1}^{N} \left(\xi\cdot M(t_r,T) \cdot v_{\mathbb{c}_r,t_r} +\xi_c \frac{\partial \Psi}{\partial x}(\ol{x}(T))\cdot  	M(T,t_r) v_{\mathbb{c}_r,t_r}\right) \varepsilon_r \le 0 
\end{array}$$
for every $(\varepsilon_1,\ldots,\varepsilon_N)\in [0,+\infty[ ^N$, which is equivalent to say that 
\bel{prodottonegativo}
 \Bigg(\left(\xi -\lambda \frac{\partial \Psi}{\partial x}(\ol{x}(t))\right)\cdot  M(T,t_r)\Bigg) \cdot v_{\mathbb{c}_r,t_r} \le 0 \qquad \forall k=1,\ldots,N,
\eeq
where we have set $\lambda:=-\xi_c(\geq 0)$.

	Now we shall utilize the invariance of the product of a solution of the adjoint  system with a solution of the variational system.  
	Let the  use  $p:[0,T]\to (\rr^n)^*$  to denote  the solution of the adjoint Cauchy problem 
	\begin{equation}\label{adj} \begin{cases} \ds p'(t)=-p(t) \frac{\partial}{\partial{x}}\left(f(x)+\sum_{1}^m g_i(x) \ol{u}(t)\right)_{x=\ol{x}(t)}  \\\ds  p(T)= \xi - \lambda\frac{\partial{\Psi}}{\partial{x}} (\ol{x}(T)). \end{cases}\end{equation}
As is well-known, one has  $\ds p(t) = \left(\xi -\lambda \frac{\partial \Psi}{\partial x}(\ol{x}(t))\right)\cdot  M(T,t)$, for all $t\in [0,T]$. Thus, in particular,  the pair $(p(\cdot),\lambda)$ verifies $\bf i)$, $\bf ii)$ and $\bf iii)$ of Theorem \ref{TeoremaPrincipale}.\footnote{Indeed $(p(\cdot),\lambda)\neq 0$, which coincides with  $\bf i)$ of Theorem \ref{TeoremaPrincipale}. Moreover the   equation in \eqref{adj} coincides with the adjoint equation  $\frac{dp}{dt} = -  \frac{\partial H}{\partial x} (\ol{x}(t),p(t),\ol{u}(t))$ in $\bf ii)$, while the initial condition  in \eqref{adj} is exactly the non-trasversality condition $\bf iii).$ }
Furthermore, by the invariance of  the  scalar  product $p(t)\cdot v(t)$on $[0,T]$, by  \eqref{prodottonegativo} we obtain 
\bel{prop1} p({t}_k)  v_{\bm{c}_k,t_k} \le 0 \qquad \forall k=1,\ldots,N.\eeq Now, specializing \eqref{prop1} to the various cases described in \eqref{notation}, we obtain
\bel{finite>}
\begin{array}{ll}\ds p({t}_k)\cdot\left( \sum_{\ell=1}^ mg_\ell(\ol{x}({{t}_k}))\big({u}^\ell - \ol{u}^\ell({{t}_k})\big)\right)\leq 0\quad &\text{if}\,\,\mathbf{c}_k={u} \in \mathfrak{V}_{ndl}\\\\
	\, p({t}_k)\cdot[g_i,g_j](\ol{x}({{t}_k}))\leq 0 \quad &\text{if} \,\,\mathbf{c_k}=(i,j) \in \mathfrak{V}_{Goh}\\\\ \,	 p({t}_k)\cdot[f,g_i](\ol{x}({t}_k))\leq 0 \quad &\text{if} \,\, \mathbf{c}=(0,i) \in \mathfrak{V}_{LC2}\\\\ \,	p({t}_k)\cdot[g,[f,g]](\ol{x}(\ol{{t}_k})) \leq 0  \quad   &\text{if} \,\, m=1\,\,\,\text{and}\,\,\,  \{\mathbf{c_k}\} = \{(1,0,1)\} =\mathfrak{V}_{LC3}.\end{array}
\eeq
Furthermore, in view of Remark \ref{imprem} we also get 

\bel{finite>}
\begin{array}{ll}
	\, p({t}_k)\cdot[g_i,g_j](\ol{x}({{t}_k}))\geq 0 \quad &\text{if} \,\,\mathbf{c_k}=(i,j) \in \mathfrak{V}_{Goh}\\\\ \,	 p({t}_k)\cdot[f,g_i](\ol{x}({t}_k))\geq 0 \quad &\text{if} \,\, \mathbf{c}=(0,i) \in \mathfrak{V}_{LC2}\end{array}
\eeq
so that we can improve \eqref{finite>} up to obtain
\bel{finite}
\begin{array}{ll}\ds p({t}_k)\cdot\left( \sum_{\ell=1}^ mg_\ell(\ol{x}({{t}_k}))\big({u}^\ell - \ol{u}^\ell({{t}_k})\big)\right)\leq 0\quad &\text{if}\,\,\mathbf{c}_k={u} \in \mathfrak{V}_{ndl}\\\\
	\, p({t}_k)\cdot[g_i,g_j](\ol{x}({{t}_k}))= 0 \quad &\text{if} \,\,\mathbf{c_k}=(i,j) \in \mathfrak{V}_{Goh}\\\\ \,	 p({t}_k)\cdot[f,g_i](\ol{x}({t}_k))= 0 \quad &\text{if} \,\, \mathbf{c}=(0,i) \in \mathfrak{V}_{LC2}\\\\ \,	p({t}_k)\cdot[g,[f,g]](\ol{x}(\ol{{t}_k})) \leq 0  \quad   &\text{if} \,\, m=1\,\,\,\text{and}\,\,\,  \{\mathbf{c_k}\} = \{(1,0,1)\} =\mathfrak{V}_{LC3}.\end{array}
\eeq 

Hence, the restriction to the instants $t_1,\ldots,t_N$ of  
 $\bf iv)$, $\bf v$, $ \bf vi)$ and $\bf vii)$  in  Theorem \ref{TeoremaPrincipale} have been proved.
	
	\subsection{Infinitely many variations} To complete  the proof  of Theorem \ref{TeoremaPrincipale}, we need to extend the validity of \eqref{finite} from a finite set of distict  instants $0<t_1,\ldots,t_k<T$ to a full-measure subset of $[0,T]$. Even though this is a standard procedure  ---based on Cantor's non\textendash{}empty  intersection theorem--- for the sake of self-consistency we will indicate the main steps of this final part of the proof.    \\

	\begin{definition}Let $X\subseteq \left(D\cap \Lambda\right)\times {\mathfrak{V}}^{\ol{u}}$ be any subset of  pairs $(t,\mathbf{c})$. We will say that multipliers $(p,\lambda) \in AC\big([0,T],(\mathbb{R}^n)^*\big)\times [0,+\infty[$ {\rm satisfies property ${\mathcal P}_X$} if the following conditions  {\bf(1)}-{\bf(6)} are verified:
		\begin{itemize}
			\item[\bf(1)] $p$ is a solution on $[0,T]$ of  the differential inclusion $$\frac{dp}{dt} = -  \frac{\partial H}{\partial x} (\ol{x}(t),p(t),\ol{u}(t))\qquad a.e.\quad t\in [0,T] \,\,\,  $$
			\item[\bf(2)] One has  $$ p(T)\in -\lambda \frac{\partial^C\Psi}{\partial{x}} \big(\ol{x}(T)\big)  -C^{\perp}   .$$
			\item[\bf(3)] For every $(t,\mathbf{c}) \in X$, if $\mathbf{c}=u\in\mathfrak{V}_{ndl}$, then $$  p(t)\,\left( {f(\ol{x}(t))} + \sum\limits_{i=1}^m g_i(\ol{x}(t)){u}^i\right)\le   
			p(t)\,\left(f(\ol{x}(t))+\sum\limits_{i=1}^m g_i(\ol{x}(t))\ol{u}^i(t)\right).
			$$

			\item[\bf(4)] For every $(t,\mathbf{c}) \in X$ and if $\mathbf{c}=(i,j)$, with $1\leq i<j\leq m$,  then $$\displaystyle0 =  p(t)\,\cdot [g_i,g_j](\ol{x}(t)).$$
				\item[\bf(5)]  For every $(t,\mathbf{c}) \in X$ and if $\mathbf{c}=(0,i)$, with $1\leq i\leq m$,   then $$\displaystyle 0 =  p(t)\,\cdot [f,g_i](\ol{x}(t)).$$
			\item[\bf(6)] For every $(t,\mathbf{c}) \in X$, if $m=1$, $g:=g_1$, $\mathbf{c} = (1,0,1)$, and $f,g$ are of class $C^{2}$ near $\ol{x}(t)$,  then
			$$
			{0\ge   p(t)\,\cdot [g,[f,g]](\ol{x}(t)) }.
			$$
		\end{itemize}

		Finally, let us  define the subset $\Theta(X)\subset AC\big([0,T],(\mathbb{R}^n)^*\big)\times  [0,+\infty[$ as $$\Theta(X):=\left\{\begin{aligned} & (p,\lambda) \in AC\big([0,T],(\mathbb{R}^n)^*\big)\times  [0,+\infty[\,:\, |(p(T),\lambda)|=1,\,\\ & (p,\lambda) \text{ verifies the property } {\mathcal P}_X\end{aligned}\right\}.\footnotemark$$
	\end{definition}
	\footnotetext{The norm inside the parentheses is the operator norm of $\left({\Bbb R}^n\times\Bbb R\right)^*$.}
\skip0.8truecm
	{
		Clearly Theorem \ref{TeoremaPrincipale} is proved as soon as we are able to show that set $\Theta\Big(\left(D\cap \Lambda\right)\times {\mathfrak{V}}^{\ol{u}}\Big)$ is non\textendash{}empty.  
		By \eqref{finite} we already know that  that $\Theta(X)\neq\emptyset$  whenever $X$ comprises $N$ couples $(t_k,\mathbf{c}_k)\in \left(D\cap \Lambda\right)\times {\mathfrak{V}}^{\ol{u}}$ such that $0<t_1<\ldots <t_N<0$.
		It can be  shown  that if we allow  $X$ to have the general  form 
		$X=\Big\{ (t_k,\mathbf{c}_k)\in \left(D\cap \Lambda\right)\times {\mathfrak{V}}^{\ol{u}}, \,\,\,0<t_1\leq t_2\ldots \leq t_N<0\Big\},$ then
		$X$ is still not empty. {{This is clearly true for the continuity of vector fields involved in the problem and their Lie Brackets (see e.g. \cite{noi} for details) {\footnote{For every $t_k$ one can find a sequence $(t_{n,k})_{n\in\mathbb{nn}}$ such that $t_{n,k}<t_{n+1,k}$ for all $n$ and $t_{n,k} \to t_k+$  and argue taking the limit of the points  $t_{n,k}$ and the corresponding multipliers $p(t_k)$ .}}}}.
		To conclude the proof of Theorem  \ref{TeoremaPrincipale}, notice that $$\Theta(X_1\cup X_2)=\Theta(X_1)\cap \Theta(X_2),
		\quad \forall X_1,X_2\subseteq D \times \mathfrak{V},$$ so that
		\begin{equation}\label{lemint}\Theta\Big(\left(D\cap \Lambda\right)\times {\mathfrak{V}}^{\ol{u}}\Big)=\bigcap\limits_{\substack{X \subseteq \left(D\cap \Lambda\right)\times {\mathfrak{V}}^{\ol{u}} \\X \text{ finite }}} \Theta(X).\end{equation}
		So we can deduce that $\Theta\Big(\left(D\cap \Lambda\right)\times {\mathfrak{V}}^{\ol{u}}\Big)\neq \emptyset$ by invoking  Cantor's theorem, which says  that the  intersection  of an infinite family of sets  is non-empty  provided every finite  intersection of sets of the  family is non-empty. The proof of Theorem 
		\ref{TeoremaPrincipale} is concluded.

\begin{remark}[Possible generalizations to non-smooth systems] In \cite{noi} the authors have investigated the problem of establishing higher order necessary conditions for minima under weak hypotheses of regularity on the vector fields of the dynamics. In particular, Goh and Legendre-Clebsch conditions have been obtained by making use of {\it set-valued Lie brackets} of Lipschitz continous vector fields. We conjecture that  these kinds of results can be extended to the setting considered in the present paper.    \end{remark}

	\vskip5truemm

	\section{Acknowledgements}
The first author is supported by MathInParis project by Fondation Sciences\\ mathématiques de Paris (FSMP), funding from the European Union’s Horizon 2020 research and innovation programme, under the Marie Skłodowska-Curie grant agreement No 101034255. \\
She is also supported by Sorbonne Universite, being affiliated at Laboratoire Jacques Louis Lions (LJLL). \\
Both the author during the writing of this paper were members of the Gruppo Nazionale per
l'Analisi Matematica, la Probabilit\`{a} e le loro Applicazioni (GNAMPA) of the
Istituto Nazionale di Alta Matematica (INdAM). \\
In particular they were members of "INdAM -GNAMPA Project 2024", codice CUP E53C23001670001, "Non-smooth optimal control problems". \\
The second author is still a GNAMPA member. \\
The second author are also members of:
PRIN 2022, Progr-2022238YY5-PE1, "Optimal control problems: analysis, approximations and applications".\\

\end{document}